\documentclass[11pt]{article}
\usepackage[english]{babel}

\usepackage{latexsym}
\usepackage{amsmath}
\usepackage{amsfonts}
\usepackage{mathrsfs}
\usepackage{amstext}
\usepackage{amssymb}
\usepackage{amsthm}       %proof environment
\usepackage{hyperref}
\usepackage{a4wide}

\usepackage{todonotes}

\usepackage{xspace}
\usepackage{xcolor}

\newcommand\cbk{\color{black}}

\newcommand\R{\ensuremath{\mathbb{R}}}

\newcommand\N{\ensuremath{\mathbb{N}}}
\newcommand\eps{\ensuremath{\varepsilon}}

\allowdisplaybreaks

\newtheorem{theorem}{Theorem}[section]
\newtheorem{prop}[theorem]{Proposition}
\newtheorem{cor}[theorem]{Corollary}

\newtheorem{defn}[theorem]{Definition}
\newtheorem{rem}[theorem]{Remark}
\newtheorem{lem}[theorem]{Lemma}

\usepackage[margin=1.2in]{geometry}

\usepackage{graphicx, fullpage}	
\begin{document}

\title{Ramification of multiple eigenvalues for the Dirichlet-Laplacian in perforated domains}
\author{\ }

\author{
Laura Abatangelo\thanks{Dipartimento di Matematica, Politecnico di Milano, P.zza Leonardo da Vinci 32, 20133 Milano, Italy \texttt{laura.abatangelo@polimi.it}} ,\ 
 Corentin L\'ena\thanks{Institut de Math\'ematiques, B\^atiment UniMail
Rue Emile-Argand 11, 2000 Neuch\^atel, Switzerland \texttt{corentin.lena@unine.ch}} ,
\ and Paolo Musolino\thanks{Dipartimento di Scienze Molecolari e Nanosistemi, Universit\`a Ca' Foscari Venezia, via Torino 155, 30172 Venezia Mestre, Italy \texttt{paolo.musolino@unive.it}}}

%\date{20220804\_multiple\_eige\_hole}

\date{August 24, 2022}

\maketitle

\noindent
{{\bf Abstract.} Taking advantage from the so-called {\em Lemma on small eigenvalues} by Colin de Verdi\`{e}re, we study ramification for multiple eigenvalues of the Dirichlet Laplacian in bounded perforated domains. The asymptotic behavior of multiple eigenvalues turns out to depend on the asymptotic expansion of suitable associated eigenfunctions.

We treat the case of planar domains in details, thanks to the asymptotic expansion of a generalization of the so-called $u$-capacity which we compute in dimension 2.
In this case multiple eigenvalues are proved to split essentially by different rates of convergence of the perturbed eigenvalues or by different coefficients in front of their expansion if the rate of two eigenbranches turns out to be the same.   
}

\vspace{11pt}


\vspace{11pt}


\section{Introduction}

\noindent The present paper deals with multiple eigenvalues for Dirichlet Laplacian in bounded domains {with small holes}. Let $\Omega$ be a bounded open set in $\mathbb{R}^{d}$, $d\geq2$, the considered eigenvalue problem is  
\begin{equation}\label{eq:eige}
 \begin{cases}
  -\Delta u = \lambda u &\text{in }\Omega\, ,\\
  u=0 &\text{on }\partial \Omega\, .
 \end{cases}
\end{equation}
 In what follows, $\N$ denotes the set of natural numbers $\{0,1,2,\ldots\}$ and $\N^\ast \equiv \N \setminus \{0\}$. From classical results in spectral theory, the Dirichlet Laplacian admits a sequence of real eigenvalues tending to infinity
\[
0<\lambda_1(\Omega)\leq\lambda_2(\Omega)\leq \dots \leq \lambda_N(\Omega)\leq \dots
\]
where every eigenvalue is repeated as many times as its multiplicity and in particular the first one is  simple if $\Omega$ is connected.
{The dependence of Laplace operator's spectrum upon domain perturbations has been long investigated, with particular attention to the case of sets with small perforations.

Samarski\u\i \   \cite{Sa48} considered the variation $\delta \lambda_N$ of an eigenvalue $\lambda_N$ for the Dirichlet-Laplacian when a small set $\omega_\varepsilon$ is removed from a subset $\Omega$ of $\mathbb{R}^3$ and obtained the estimate
\[%\begin{equation}\label{eq:Sam}
\delta \lambda_N \leq 4\pi \kappa_N^2 \mathrm{Cap}_\Omega(\omega_\varepsilon)+O(\mathrm{Cap}_\Omega(\omega_\varepsilon)^2)\, ,
\]%\end{equation}
where $\mathrm{Cap}_\Omega(\omega_\varepsilon)$ is the {\it standard capacity} of $\omega_\varepsilon$ in $\Omega$, $\kappa_N$ is the maximum value of the $N$-th normalized eigenfunction on $\overline{\omega_\varepsilon}$ (see  Maz'ya, Nazarov, and Plamenevski\u\i \   \cite{MaNaPl84}). Here we recall that, if we consider a bounded, connected open set $\Omega$   of $\mathbb{R}^d$, then for every compact subset $K$ of $\Omega$, the {\it standard capacity} of $K$ in $\Omega$ is defined as
\[%\begin{equation} \label{eq:cap}
\mathrm{Cap}_{\Omega}(K)\equiv \mathrm{inf} \bigg\{\int_{\Omega} |\nabla f|^2\, dx \colon f \in H^1_0(\Omega)\ \mathrm{and}\ f-\eta_K \in H^1_0(\Omega \setminus K)\bigg\}\, ,
\]%\end{equation}
where $\eta_K$ is a fixed smooth function such that $\mathrm{supp}\, \eta_K \subseteq \Omega$ and $\eta_K \equiv 1$ in a neighborhood of $K$.
Rauch and Taylor  \cite{RaTa75} investigated the behavior of the eigenvalues and eigenfunctions of the Laplacian in a domain $\Omega$ where a ``thin'' set is removed. 
In a series of papers (see ,  {\it e.g.}, \cite{Oz80, Oz81a, Oz81b, Oz82a, Oz82b}), Ozawa  computed several asymptotic expansions for the eigenvalues {of} the Laplacian, under many different boundary conditions, when a small perforation is made. For example, Ozawa has shown in \cite{Oz81b} that
if $n=2$  then 
\[%\begin{equation}\label{eq:Ozawa}
\lambda_N(\Omega \setminus (\varepsilon \overline{\mathbb{B}_2(0,1)}))=\lambda_N(\Omega)-2\pi(\log \varepsilon)^{-1} (u_N(0))^2+O((\log \varepsilon)^{-2}) \qquad \text{as $\varepsilon \to 0^+$}\, ,
\]%\end{equation}
where $\mathbb{B}_2(0,1)$ is the unit ball in $\mathbb{R}^2$,  $\lambda_N(\Omega)$ is a simple eigenvalue for the Dirichlet-Laplacian in $\Omega$ and $u_N$ a corresponding $L^2(\Omega)$-normalized eigenfunction.

A detailed analysis of boundary value problem in singularly perturbed domain has been performed by Maz'ya, Nazarov, and Plamenevski\u\i, who considered also eigenvalue problems (see, e.g., \cite{MaNaPl84} and \cite[Chapter 9]{MaNaPl00i}). For example, in the three dimensional case, they have shown in \cite{MaNaPl84} that for the first eigenvalue of the Laplacian with Dirichlet condition we have
\[%\begin{equation}\label{eq:Mazya}
\begin{split}
\lambda_1(\Omega \setminus (\varepsilon \overline{\omega}))=&\lambda_1(\Omega)+4\pi \mathrm{Cap}(\omega)(u_1(0))^2 \varepsilon +[4\pi u_1(0)\mathrm{Cap}(\omega)]^2 \\
& \times \Big\{ \Gamma(0)+\frac{u_1(0)}{4\pi}\int_\Omega u_1(x)|x|^{-1}\, dx\Big\}\varepsilon^2+O(\varepsilon^3)\qquad \text{as $\varepsilon \to 0^+$}\, ,
\end{split}
\]%\end{equation}
where $u_1$ is a corresponding $L^2(\Omega)$-normalized eigenfunction in $\Omega$, $\mathrm{Cap}(\omega)$ the harmonic capacity of $\omega$ and $\Gamma$ is a function defined through an auxiliary boundary value problem.

Moreover, techniques based on potential theory and integral operators for the study of eigenvalues of the Laplacian in perforated domains have been exploited for example in Ammari, Kang, and Lee \cite{AmKaLe09} and in Lanza de Cristoforis \cite{La12}.

The behavior of the spectrum of the Laplacian under removal of ``small'' sets in the Riemannian setting has been studied by many authors: as an example, we mention the works by Besson \cite{Be85}, Chavel \cite{Ch84}, Chavel and Feldman \cite{ChFe88}, Colbois and Courtois \cite{CoCo91}, Courtois \cite{Courtois1995}. 

If on one hand, most of the above mentioned results deals with simple eigenvalues, less is known for the case of multiple eigenvalues.

A general study on perturbation problems for eigenvalues can be found in the monograph by Kato \cite{Ka95}, which contains also some results for multiple eigenvalues.

In  Nguyen \cite{Ng94}, the author investigated the asymptotic expansion of the multiple eigenvalues and eigenfunctions for boundary value problems in a domain with a small hole. More precisely, he studied the bifurcation of a double or triple eigenvalue in a smooth bounded domain in $\mathbb{R}^3$ where a small ball of radius $\eps$ is removed.

Asymptotic expansions for eigenvalues, both simple and multiple, and eigenfunctions of the Neumann Laplacian in a three-dimensional domain with a small hole are obtained in  Nazarov and Sokolowski \cite{NaSo08}, Laurain,  Nazarov, and Sokolowski \cite{LaNaSo11}, and  Novotny  and Sokolowski \cite{NoSo13} in terms of the eigenvalues of an auxiliary matrix.

Asymptotic expansions for multiple eigenvalues of regularly perturbed domains are obtained in Bruno and Reitich \cite{BrRe01}. Concerning regular perturbations, analyticity results for symmetric functions of possibly multiple eigenvalues upon regular domain perturbations can be found in Lamberti and Lanza de Cristoforis \cite{LaLa02, LaLa04, LaLa07}.

We note that a fundamental tool in the study of perturbation problems for the eigenvalues of the Laplacian is the standard capacity. Indeed, as is well known, if we remove a compact subset $K$ of zero capacity, the spectrum of the Dirichlet-Laplacian on the bounded domain $\Omega$ does not change (see ,  {\it e.g.}, Rauch and Taylor \cite{RaTa75}).
{Moreover, if we denote by
\[
0<\lambda_1(\Omega \setminus K)<\lambda_2(\Omega \setminus K)\leq \dots \leq \lambda_N(\Omega \setminus K)\leq \dots
\]
the sequences of the eigenvalues of the Dirichlet-Laplacian in $\Omega \setminus K$, then Rauch and Taylor \cite{RaTa75} also proved that} the $N$-th eigenvalue $\lambda_N(\Omega\setminus K)$ of the Dirichlet-Laplacian in $\Omega \setminus K$ is close to $\lambda_N(\Omega)$ if  the capacity $\mathrm{Cap}_\Omega(K)$ of $K$ in $\Omega$ is small. 

On one hand, the result by Rauch and Taylor \cite{RaTa75} can be seen as a continuity result for eigenvalues with respect to the capacity. On the other hand, Courtois \cite{Courtois1995} has obtained a higher regularity: he has shown that if $K\subseteq \Omega$ is compact and $\mathrm{Cap}_\Omega(K)$ is small then the function
\[%\begin{equation}\label{eq:difflambda}
\lambda_N(\Omega \setminus K)-\lambda_N(\Omega)
\]%\end{equation}
is in some sense differentiable with respect to $\mathrm{Cap}_\Omega(K)$. }

A first improvement of Courtois' results has been achieved by \cite{AbFeHiLe2019} and later on by \cite{AbBoLeMu}. In the first paper, for simple eigenvalues the authors
establish a sharp relation between the vanishing order of a Dirichlet
eigenfunction at a point and the leading term of the  asymptotic expansion of the corresponding
Dirichlet eigenvalue variation, as a removed compact set concentrates
at that point. We recall indeed that a point has zero capacity. If, for example, $\lambda_N(\Omega)$ is simple, then one can  replace the capacity $\mathrm{Cap}_\Omega(K)$ by the so-called $u_N$-capacity $\mathrm{Cap}_{\Omega}(K,u_N)$, $u_N$ being a suitable associated eigenfunction (see equation \eqref{eq:ucap} below). In this way one can obtain more refined asymptotic expansions of the difference $\lambda_N(\Omega \setminus K)-\lambda_N(\Omega)$ (on this topic, see also \cite{CoBe06}).
In the second cited paper \cite{AbBoLeMu}, the authors provide
the asymptotic behavior of  $u$-capacities of small sets of type $\eps{\overline{\omega}}$, with $\omega$ a sufficiently regular open bounded connected set of $\R^2$ and $u$  a function analytic in a neighborhood of the origin. Consequently, they are able to deduce asymptotic behaviors for simple eigenvalues when a small set of this type is removed from the original domain. In particular, this analysis sheds some light on the dependence of these asymptotics on the hole's shape $\omega$.

\bigskip

In the present paper we rather consider multiple eigenvalues. We aim at investigating whether multiple eigenvalues can be splitted in different branches if we perturb {the} initial domain by removing a small hole around a point, namely the origin. In \cite[Theorem 1.2]{Courtois1995} Courtois gave a first response in this direction. We recall his result in the following  
\begin{theorem}{\rm \cite[Theorem 1.2]{Courtois1995}}\label{t:courtois1.2}
  Let $X$ be a compact Riemannian manifold with or without boundary of dimension greater than or equal to 2.  Let
  $\lambda:=\lambda_{N}=\ldots=\lambda_{N+k-1}$ be a Dirichlet eigenvalue of
  $X$ with multiplicity $k$.  There exist a function
  $r:\,\R^+\to \R^+$ such that $\lim_{t\to0} r(t)=0$ and  a positive
  constant $\eps_N$, such that, for any compact  subset
  $A$ of $X$, if $\mbox{\rm Cap}_X(A)\le\eps_N$, then
 \[%\begin{equation}\label{eq:7}
 |\lambda_{N+j}(X\setminus A)-\lambda_{N+j} - \mbox{\rm Cap}_X (A)\cdot
 \mu_A(u_{N+j}^2)| \le \mbox{\rm Cap}_X (A)\cdot r(\mbox{\rm Cap}_X (A))
\]
 where $\mu_A$ is a finite positive probability measure supported in
 $A$ defined as the renormalized singular part of $-\Delta V_A$ and
 $\{u_{N},\ldots,u_{N+k-1}\}$ is an orthonormal basis of the
 eigenspace of $\lambda$ which diagonalises the quadratic form
 $\mu_A(u^2)$ according to the increasing order of its eigenvalues.\end{theorem}
Theorem \ref{t:courtois1.2} above provides a sharp
asymptotic expansion  of $\lambda_{N+j}(X\setminus A) -\lambda_{N+j}$ as $\mbox{\rm Cap}_X(A)\to 0$ only
if $\mu_A(u_{N+j}^2)\not\to0$, but in general it is just an estimate for $\lambda_{N+j}(X\setminus A) - \lambda_{N+j}$ 
when
$\mu_A(u_{N+j}^2)\to0$.
Moreover, even when $\mu_A(u_{N+j}^2)\not\to0$ the possible splitting of a multiple eigenvalue is hidden in the involved quadratic form $\mu_A(u^2)$. (We note that in the present paper, with an abuse of notation, we denote by the same symbol a bilinear form and the corresponding quadratic form.)

To introduce our main results, we need to specify the functional context where we are working.

\subsection{Functional setting and basic problem}

Let us give a more convenient reformulation of the eigenvalue problem \eqref{eq:eige}. We consider the quadratic form defined by
\begin{equation}\label{eq:qf}
	q_\Omega(u)\equiv\int_\Omega |\nabla u|^2 \, dx
\end{equation}
for $u\in H_0^1(\Omega)$. We denote by $\langle\cdot\,,\cdot\rangle$ the usual scalar product in $L^2(\Omega)$ and by $\|\cdot\|$ the associated norm. It is well known that the spectrum  of $q_\Omega$ with respect to $\langle\cdot\,,\cdot\rangle$ consists of positive eigenvalues of finite multiplicity tending to infinity. We denote the sequence of these eigenvalues, repeated according to their multiplicity, by $\{\lambda_i(\Omega)\}_{i\ge1}$. 

We are interested in the effect on the eigenvalues of imposing an additional Dirichlet condition on a small compact set $K\subseteq\Omega$. This can intuitively be described as creating a small hole $K$ in $\Omega$. To make our meaning clear, let us recall the following definition from \cite{AbFeHiLe2019}.

\begin{defn} Let $K\subseteq\Omega$ be compact and let $\{K_\eps\}_{\eps>0}$ be a family of compact sets contained in $\Omega$. We say that $\{K_\eps\}$ \emph{concentrates to $K$} when, for every open set $U$ satisfying $K\subseteq U\subseteq\Omega$, there exists $\eps_U>0$ such that, for all $\eps<\eps_U$, $K_\eps\subseteq U$. 
\end{defn}
We now fix $K\subseteq \Omega$ a compact set, such that $\mathrm{Cap}_\Omega(K)=0$, and $\{K_\eps\}_{\eps>0}$ a family of compact sets concentrating to $K$. We define, for $\eps>0$, the perforated domain $\Omega_\eps\equiv\Omega\setminus K_\eps$.  The following result is easy to prove. 
\begin{prop} \label{prop:convEV} For all $i\ge1$, $\lambda_i(\Omega_\eps)\to\lambda_i(\Omega\setminus K)$ as $\eps\to0$ and $\lambda_i(\Omega\setminus K)=\lambda_i(\Omega)$.
\end{prop}
To simplify the notation, we set $\lambda_i\equiv\lambda_i(\Omega)$ and $\lambda_i^\eps\equiv\lambda_i(\Omega_\eps)$.
We wish to estimate the difference $\lambda_i^\eps-\lambda_i$. To that end, we recall a definition from \cite{AbFeHiLe2019}, inspired by \cite{CoBe06,Courtois1995}.
\begin{defn}\label{def:ucap} Given a function $u\in H^1_0(\Omega)$, the $u$-capacity of a compact set $K\subseteq \Omega$ is 
 \begin{equation} \label{eq:ucap}
\mathrm{Cap}_{\Omega}(K,u)\equiv \mathrm{inf} \bigg\{\int_{\Omega} |\nabla f|^2\, dx \colon f \in H^1_0(\Omega)\ \mathrm{and}\ f-u \in H^1_0(\Omega \setminus K)\bigg\}\,. 
\end{equation}
The infimum in \eqref{eq:ucap} is achieved by a unique function $V_{K,u}\in H_0^1(\Omega)$, so that
\[
\mathrm{Cap}_\Omega(K,u)=\int_{\Omega}|\nabla V_{K,u}|^2 \, dx\, .
\]
We call $V_{K,u}$ the \emph{potential} associated with $u$ and $K$.

Furthermore, $V_{K,u}$ is the unique weak solution of  the Dirichlet problem
\begin{equation}\label{eq:dircapKu}
\left\{
\begin{array}{ll}
\Delta V_{K,u}=0&\text{ in }\Omega \setminus K\,,\\
V_{K,u}=0&\text{ on } \partial\Omega\,,\\
V_{K,u}=u &\text{ on }K\,,
\end{array}
\right.
\end{equation}
where, by \emph{weak solution}, we mean  that $V_{K,u}\in H^1_0(\Omega)$, $u-V_{K,u}\in H_0^1(\Omega\setminus K)$ and $\int_\Omega \nabla V_{K,u}\cdot\nabla \varphi=0$ for all $\varphi\in H_0^1(\Omega\setminus K)$. 
\end{defn}

\begin{rem} Let us note that the Sobolev space $H^1_0(\Omega\setminus K)$ can be seen as a subspace of $H^1_0(\Omega)$. Indeed, the quadratic form $q_\Omega$ of Equation \eqref{eq:qf} endows $H_0^1(\Omega)$ with a Hilbert space norm. The space $H^1_0(\Omega\setminus K)$ can then be identified with the closure of the subspace $C_c^\infty(\Omega\setminus K)$, consisting of smooth functions compactly supported in $\Omega\setminus K$. More explicitly, this identification associates to $f\in H^1_0(\Omega\setminus K)$ its extension $g$ to $\Omega$ obtained by setting $g=f$ in $\Omega\setminus K$ and $g=0$ in $K$.  We systematically perform this identification in the paper, and in particular we use the same notation for an element of $H^1_0(\Omega\setminus K)$ and the element of $H^1_0(\Omega)$ obtained after extension.
\end{rem}

\begin{rem} \label{rem:geom} Definition \ref{def:ucap} has a geometric interpretation.  According to the previous remark, $H_0^1(\Omega\setminus K)$ can be seen as a subspace of $H_0^1(\Omega)$, closed for the Hilbert space norm given by the quadratic form $q_\Omega$.
%The quadratic form $q_\Omega$ of Equation \eqref{eq:qf} endows $H_0^1(\Omega)$ with a Hilbert space structure, for which $H_0^1(\Omega\setminus K)$ is a closed subspace.
 Then, $\mathrm{Cap}_\Omega(K,u)$ is the square of the $q_\Omega$-distance of $u$ to $H_0^1(\Omega\setminus K)$ and $u-V_{K,u}$ is the $q_\Omega$-orthogonal projection of $u$. The existence and uniqueness of $V_{K,u}$, as well as its characterization as the unique weak solution of \eqref{eq:dircapKu}, follow immediately from the properties of the orthogonal projection on a closed subspace of a Hilbert space.

Note also that Definition \eqref{eq:ucap} can be extended to $H^1(\Omega)$ functions, by setting, for any $u \in H^1(\Omega)$, \[\mathrm{Cap}_\Omega(K,u) \equiv \mathrm{Cap}_\Omega(K,\eta_K u)\] where $\eta_K$ is a fixed smooth function such that $\mathrm{supp}\, \eta_K \subseteq \Omega$ and $\eta_K \equiv 1$ in a neighborhood of $K$.
\end{rem}
In order to study multiple eigenvalues, we will need the following generalization of the $u$-capacity.
\begin{defn}\label{def:uvcap} Given $u,v\in H_0^1(\Omega)$, the $(u,v)$-capacity of a compact set $K\subseteq \Omega$ is 
\begin{equation*}
	\mathrm{Cap}_\Omega(K,u,v)\equiv\int_\Omega\nabla V_{K,u}\cdot\nabla V_{K,v}\, dx.
\end{equation*}
\end{defn}
\begin{rem} \label{rem:geombis} As it is done in Remark \ref{rem:geom}, we extend Definition \ref{def:uvcap} of $\mathrm{Cap}_\Omega(K,u,v)$ to functions $u,v\in H^1(\Omega)$, by setting 
\[\mathrm{Cap}_\Omega(K,u,v) \equiv \mathrm{Cap}_\Omega(K,\eta_K u, \eta_K v)
\] where $\eta_K$ is a fixed smooth function such that $\mathrm{supp}\, \eta_K \subseteq \Omega$ and $\eta_K \equiv 1$ in a neighborhood of $K$.
\end{rem}

Let us now fix an eigenvalue $\lambda_N$, of multiplicity $m$, for Problem \eqref{eq:eige}. Let us denote by $E(\lambda_N)$ the associated eigenspace. According to Proposition \ref{prop:convEV}, for $i\in\{1,\dots,m\}$,
\begin{equation*}
	\lambda_{N+i-1}^\eps\to \lambda_N \mbox{ as } \eps\to0.
\end{equation*} 
We therefore have $m$ eigenvalue branches departing from the multiple eigenvalue $\lambda_N$. We will prove the following result about their asymptotic behavior.
\begin{theorem}\label{thm:approxEVs} For $i\in\{1,\dots,m\}$,
\begin{equation}\label{eq:asymptEV}
	\lambda_{N+i-1}^\eps=\lambda_N+\mu_i^\eps+o(\chi_\eps^2) \mbox{ as }\eps\to0,
\end{equation}
where
\begin{equation*}
	\chi_\eps^2\equiv\sup\{\mathrm{Cap}_\Omega(K_\eps,u)\,:\,u\in E(\lambda_N)\mbox{ and } \|u\|=1\}
\end{equation*}
and $\{\mu_i^\eps\}_{i=1}^m$ are the eigenvalues of the quadratic form $r_\eps$ defined, for $u,v\in E(\lambda_N)$, by
\begin{equation*}
	r_\eps(u,v)\equiv\int_\Omega\nabla V_{K_\eps,u}\cdot\nabla V_{K_\eps,v}\, dx-\lambda_N\int_\Omega V_{K_\eps,u}V_{K_\eps,v}\, dx.
\end{equation*}
\end{theorem}

Theorem \ref{thm:approxEVs} is a generalization of \cite[Theorem 1.4]{AbFeHiLe2019}, the latter applying only to simple eigenvalues. It is similar to \cite[Theorem 1.2]{Courtois1995}, which involves the standard capacity rather than the $u$-capacity. However, Theorem \ref{thm:approxEVs} suffers from several limitations.  First of all, the quantities $\chi_\eps$, $\{\mu_i^\eps\}$ can be difficult to compute explicitly. Furthermore, it can happen that, for some $i$, $\mu_i^\eps=o(\chi_\eps^2)$, in which case \eqref{eq:asymptEV} reduces to the estimate $\lambda_{N+i-1}^\eps-\lambda_N=o(\chi_\eps^2)$. To progress further, we will focus on regular compact sets concentrating to a point, describe the local behavior of eigenfunctions, and determine asymptotic expansions for the $(u,v)$-capacity. For definiteness, and in view of concrete applications,  we will give the final result in the two-dimensional case ($d=2$). However, our method could be extended to higher dimensions.

\subsection{Refinement using analyticity}

In this section, we fix a point $x_0$ in $\Omega\subseteq\R^d$. Without loss of generality, we can assume $x_0=0$.
Let us recall some basic facts about Laplacian eigenfunctions. Let $\lambda$ be an eigenvalue of Problem \eqref{eq:eige} and $u$ an associated eigenfunction. Since the differential operator $-\Delta-\lambda$ is elliptic with analytic coefficients, $u$ is real-analytic in $\Omega$ (see, {\it e.g.}, \cite[Theorem 7.5.1]{Ho69}). Using the standard multi-index notation, this means in particular that
\begin{equation*}
	u(x)=\sum_{ \beta\in \N^d}\frac1{ \beta!}D^{ \beta} u(0)x^{ \beta}
\end{equation*} 
for $|x|$ small enough. This can be rewritten as 
\begin{equation*}
	u(x)=\sum_{k\in\N}P_k(x)
\end{equation*} 
where 
\begin{equation*}
	P_k(x)\equiv\sum_{ |\beta|=k}\frac1{ \beta!}{ D^{ \beta} u(0)}x^{ \beta}
\end{equation*}
is a homogeneous polynomial of degree $k$ in $d$ variables. We call the smallest $k$ such that $P_k\neq0$ the \emph{order of vanishing} (or more briefly the \emph{order}) of $u$ and we denote it by $\kappa(u)$. We denote the polynomial $P_{\kappa(u)}$ by $u_\#$ and call it the \emph{principal part} of $u$. Note that these definitions make sense whenever $u$ is  real-analytic in a neighborhood of $0$, and we will use them as soon as this condition is satisfied. When $u$ is an eigenfunction, we have an additional property: it follows immediately from the eigenvalue equation $\Delta u+\lambda u=0$ that $\Delta u_\#=0$ (i.e. the homogeneous polynomial $u_\#$ is harmonic).

We now return to the situation described in the previous section, with $\lambda_N$ an eigenvalue of multiplicity $m$ for Problem \eqref{eq:eige} and $E(\lambda_N)$ the associated eigenspace. We will use the following results, proved in Appendix \ref{app:DecompES}.

\begin{prop}\label{prop:DecompES} There exists a decomposition of $E(\lambda_N)$ into a sum of orthogonal subspaces
\[E(\lambda_N)=E_1\oplus\dots\oplus E_p\] 
and an associated finite decreasing sequence of integers
\[k_1>\dots>k_p\ge0\]
such that, for all $1\le j \le p$, a function in $E_j\setminus\{0\}$ has the order of vanishing $k_j$ at $0$. In addition, such a decomposition is unique. We call it the \emph{order decomposition} of $E(\lambda_N)$. 
\end{prop}

\begin{prop}\label{prop:MaxDim} Let $E=E_1\oplus\dots\oplus E_p$, be the order decomposition of Proposition \ref{prop:DecompES}. Then the dimension of $E_j$ is at most the dimension of the space of spherical harmonics in $d$ variables of degree $k_j$ (see, {\it e.g.}, \cite[pp. 159--165]{Berger1971}). Explicitly,
\[\mbox{dim}(E_{j})\le \binom{k_j + {
 d} - 2}{k_j} + \binom{k_j + {d} - 3}{k_j-1}.
\] 
\end{prop}

\begin{rem} \label{rem:Dim} As a consequence of Propositions \ref{prop:DecompES} and \ref{prop:MaxDim}, 
\begin{itemize}
\item if $k_p=0$ (that is to say if $E$ contains an eigenfunction which does not vanish at $ 0$), $\mbox{dim}(E_p)=1$, whatever the dimension $d$ is;
\item in the case $d=2$, $\mbox{dim}(E_{j})\le 2$ for all $1\le j\le p$.
\end{itemize}
\end{rem}

 Our analysis of the $(u,v)$-capacity is of independent interest and will yield stronger results than   strictly needed for studying the behavior of eigenvalues. We will need some additional regularity of the data. We fix an $\alpha\in]0,1[$ and make the following assumptions.
\begin{itemize}
	\item[(A1)] The dimension $d$ equals $2$;
	\item[(A2)] $\Omega$ is an open, bounded and connected set in $\R^2$, of class $C^{1,\alpha}$, such that $\R^2\setminus\overline\Omega$ is connected and $0\in\Omega$;
	\item[(A3)] $K_\eps=\eps\overline\omega$, where $\omega\subseteq \R^2$ satisfies Assumption (A2).
\end{itemize}
Let us insist on the fact that we will prove the subsequent results in the two-dimensional case (Assumption (A1)).  Let us  also note that (A2) and (A3) imply the existence of a positive number $\eps_0$ such that $K_\eps=\eps\overline\omega \subseteq \Omega$ for all $\eps\in ]-\eps_0,\eps_0[$. Finally, we consider functions $u,v$ defined in $\Omega$, satisfying
\begin{itemize}
	\item[(A4)] $u$ and $v$ belong to $H_0^1(\Omega)$ and are real-analytic in a neighborhood of $0$.
\end{itemize}

Thus, in the case of dimension $d=2$, we will be able to show the following theorem on the representation of $\mathrm{Cap}_\Omega(\eps\overline\omega, u,v)$ as the sum of a convergent series  (see Theorem \ref{capk}).

\begin{theorem} \label{t:capSeries} Under Assumptions (A1)--(A4), $\mathrm{Cap}_\Omega(\eps\overline\omega, u,v)$ is the sum of a convergent series. More precisely, there exist $r_0 \in \mathbb{R}$, a family of real numbers $\{c_{(n,l)}\}_{\substack{(n,l)\in\mathbb{N}^2\\\;l\le n+1}}$, and $\varepsilon_\mathrm{c}$ positive and small enough, such that the series 
\[
\sum_{n=0}^\infty\varepsilon^n\sum_{l=0}^{n+1} \frac{c_{(n,l)}\eta^l}{(r_0\eta+(2\pi)^{-1})^{l}}\] converges  absolutely for all $(\varepsilon,\eta)\in]-\varepsilon_\mathrm{c},\varepsilon_\mathrm{c}[\times]1/\log\varepsilon_\mathrm{c},-1/\log\varepsilon_\mathrm{c}[$ and that
\begin{equation} \label{eq:capSeries}
\mathrm{Cap}_\Omega(\varepsilon\overline{\omega},u,v)=\sum_{n=0}^\infty\varepsilon^n\sum_{l=0}^{n+1} \frac{c_{(n,l)}}{(r_0+(2\pi)^{-1}\log|\varepsilon|)^{l}}
\end{equation}
for all $\varepsilon\in]-\varepsilon_{\mathrm{c}},\varepsilon_\mathrm{c}[\setminus\{0\}$. 
Moreover,  $c_{(0,0)}=0$ and $c_{(0,1)}= -u(0)v(0)$.
\end{theorem}
We will need more precise information on the coefficients $\{c_{(n,l)}\}_{\substack{(n,l)\in\mathbb{N}^2\\\;l\le n+1}}$ in the series \eqref{eq:capSeries}. In order to state our results, we introduce the following notation. Given $u$ satisfying assumption (A4) and such that $u(0)=0$, we define $ \mathsf u$ as the unique function, continuous in $\R^2\setminus\omega$, satisfying
\begin{equation}\label{eq:U}
\left\{
\begin{array}{ll}
\Delta  \mathsf u=0&\text{in }\R^2\setminus\overline\omega\,,\\
 \mathsf u=u_\#&\text{on } \partial\omega\,,\\
\sup_{\R^2\setminus\omega} |\mathsf u|<\infty\,, &
\end{array}
\right.
\end{equation}
where $u_\#$ is, as before, the principal part of $u$ (see problem \eqref{eq:ulk}).
Then, for $u,v$ satisfying (A4), with $u(0)=v(0)=0$, we define
\begin{equation}\label{eq:Q}
	\mathcal Q(u,v)\equiv\int_{\R^2\setminus\overline\omega}\nabla  \mathsf u\cdot\nabla  \mathsf v\, dx+\int_{\omega}\nabla u_\#\cdot\nabla v_\#\, dx.
\end{equation} 
We observe that the expression in \eqref{eq:Q} is well defined. Indeed, the functions $u_\#$ and $v_\#$ are $C^{\infty}$ in the whole of $\mathbb{R}^2$ (and thus their gradients are bounded in $\overline{\omega}$). Moreover, $\mathsf u$ and $\mathsf v$ are harmonic in $\R^2\setminus\overline\omega$ and harmonic at infinity and thus the Divergence Theorem and the decay of their radial derivatives imply that $\int_{\R^2\setminus\overline\omega}\nabla  \mathsf u\cdot\nabla  \mathsf v\, dx$ is bounded. Let us note that $\mathcal Q$ is \emph{not} a bilinear form, since the principal part $u_\#$ does not depend linearly on $u$ (and, as a consequence, neither does $ \mathsf u$). However, the restriction of $\mathcal Q$ to suitable subspaces of $E(\lambda_N)$ defines bilinear forms, as we will see in Definition \ref{def:Qj} below.

 Under vanishing assumptions on $u$ and $v$ at $0$, we deduce the validity of the following proposition on the coefficients of the series in Equation \eqref{eq:capSeries} (see Theorem \ref{thm:cepsmseries}).

\begin{prop}\label{prop:coeff}
 If $u,v$ satisfy (A4), with $u(0)=v(0)=0$, the coefficients in the series \eqref{eq:capSeries} satisfy
\begin{enumerate}
 	\item $c_{(n,l)}=0$ if $n\le\kappa(u)+\kappa(v)-1$;
 	\item $c_{(\kappa(u)+\kappa(v),0)}=\mathcal Q(u,v)$.
 \end{enumerate}
\end{prop}

From Theorem \ref{t:capSeries} and Proposition \ref{prop:coeff}, we immediately deduce the asymptotic behavior of some $(u,v)$-capacities (see also Remark \ref{rem:cepsmseries}).

\begin{cor}\label{cor:asymptCap} Let us fix $u,v$ satisfying Assumption (A4). Then,
\begin{enumerate}
	\item if $u(0)$ and $v(0)$ are non-zero,
	\begin{equation*}
		\mathrm{Cap}_\Omega(\eps\overline\omega, u,v)=\frac{2\pi u(0)v(0)}{|\log(\eps)|}+o\left(\frac1{|\log(\eps)|}\right)\mbox{ as }\eps\to 0^+;
	\end{equation*}
	\item if $u(0)=v(0)=0$,
	\begin{equation*}
		\mathrm{Cap}_\Omega(\eps\overline\omega, u,v)=\eps^{\kappa(u)+\kappa(v)}\mathcal Q(u,v)+o\left(\eps^{\kappa(u)+\kappa(v)}\right)\mbox{ as }\eps\to 0^+.
	\end{equation*}
\end{enumerate}
\end{cor}

Let us now consider again an eigenvalue $\lambda_N$ of multiplicity $m$ and $E(\lambda_N)$ the associated eigenspace. We use the order decomposition defined in Proposition \ref{prop:DecompES}.

\begin{defn}\label{def:Qj} For all $1\le j\le p$, we define $\mathcal Q_j$ on $E_j$ in the following way. If $k_j=0$, $u,v\in E_j$,
\[\mathcal Q_j(u,v)\equiv 2\pi u(0)v(0).\]
If $k_j\ge1$, $u,v\in E_j$ 
\[\mathcal Q_j(u,v)\equiv\mathcal Q(u,v).\]
It is a strictly positive (in particular non-degenerate) symmetric bilinear form on $E_j$.
\end{defn}

Using the previous definitions and results, we can describe the behavior of the eigenvalues $(\lambda_i^\eps)_{N\le i\le N+m-1}$, and more specifically give the principal part of the spectral shift $\lambda_i^\eps-\lambda_N$ for each eigenvalue branch departing from $\lambda_N$. Let us formulate the result precisely. 

In order to unify the notation, we define, for $k\in \N$ and $\eps>0$,
\begin{equation}\label{eq:scale}
	\rho_k^\eps \equiv\begin{cases}
					\frac1{|\log(\eps)|}&\mbox{ if }k=0\, ,\\
					\eps^{2k}		    &\mbox{ if }k\ge1\, .
				\end{cases}
\end{equation}
The functions $\{\eps\mapsto\rho_k^\eps\}_{k\ge0}$ form a so-called \emph{asymptotic scale}: they are continuous and positive in $]0,+\infty[$, and as $\eps\to0^+$, $\rho_0^\eps\to0$ and $\rho_{k+1}^\eps=o(\rho_k^\eps)$.  We will find an equivalent of $\lambda_i^\eps-\lambda_N$ on that scale. Note that according to Lemma \ref{l:a1}, the error term $\chi_\eps^2$ (defined by Equation \eqref{eq:chiEps}) is of the same order as $\rho_{k_p}^\eps$ (recall that $k_p$ is the smallest possible order for an eigenfunction in $E(\lambda_N)$). Note also that the particular form of the functions $\{\rho_k^\eps\}$ depends on Assumption (A1), namely $d=2$.
 
\begin{theorem} \label{thm:orderEVs} 
Let us assume that (A1)--(A3) are satisfied. For $1\le j\le p$, we write 
\[m_j\equiv\mbox{dim}(E_j),\]
so that 
\[m=m_1+\dots+m_j+\dots+m_p,\]
and we denote by 
\[0<\mu_{j,1}\le\dots\le\mu_{j,\ell}\le\dots\le\mu_{j,m_j}\]
the eigenvalues of the quadratic form $\mathcal{Q}_j$. Then, for all $1\le j\le p$ and $1\le \ell \le m_j$,
\begin{equation}\label{eq:asymptEVOrder}
\lambda_{N-1+m_1+\dots+m_{j-1}+\ell}^\eps=\lambda_N+\mu_{j,\ell}\,\rho_{k_j}^\eps+o(\rho_{k_j}^\eps)\mbox{ as }\eps\to0^+.
\end{equation}
\end{theorem}

As an illustration, let us consider a particular case.
\begin{cor} \label{cor:doubleEV} Let us assume that (A1)--(A3) are satisfied and that $\lambda_N$ has multiplicity $2$ (i.e. $m=2$). Then one of the following alternative holds.
\begin{enumerate}
\item  There exist two normalized eigenfunctions $u_1,u_2\in E(\lambda_N)\setminus\{0\}$, with respective order of vanishing $k_1,k_2$ such that $k_1>k_2$. In that case, 
\begin{equation*}
	\lambda_{N}^\eps= \lambda_{N}+\mathcal Q(u_1,u_1)\eps^{2k_1}+o(\eps^{2k_1})
\end{equation*}
and
\begin{equation*}
	\lambda_{N+1}^\eps=\lambda_{N}+\begin{cases}
								\frac{2\pi u_2(0)^2}{|\log(\eps)|}+o\left(\frac1{|\log(\eps)|}\right)&\mbox{ if } k_2=0\, ,\\
								\mathcal Q(u_2,u_2)\eps^{2k_2}+o(\eps^{2k_2})&\mbox{ if } k_2\ge1\, .
						\end{cases}
\end{equation*}

\item All eigenfunctions in $E(\lambda_N)$ have the same order of vanishing, which we denote by $k$. Let us note that necessarily $k\ge1$. In that case, let us choose eigenfunctions $u_1,u_2$ forming an orthonormal basis of $E(\lambda_N)$ and let us denote by $0<\mu_1\le\mu_2$ the eigenvalues of the symmetric and positive matrix
\begin{equation*}
		\left(\begin{array}{cc}
				\mathcal Q(u_1,u_1)& \mathcal Q(u_1,u_2)\\
				\mathcal Q(u_1,u_2)& \mathcal Q(u_2,u_2)
			\end{array}
		\right).
\end{equation*}
Then
\begin{align*}
	\lambda_N^\eps=& \lambda_{N}+\mu_1\eps^{2k}+o(\eps^{2k});\\
	\lambda_{N+1}^\eps=& \lambda_{N}+\mu_2\eps^{2k}+o(\eps^{2k}).
\end{align*}
\end{enumerate}
\end{cor}

We note that in the preceding Corollary \ref{cor:doubleEV}, item 2, the splitting of the two branches does not necessarily occur since the two eigenvalues $\mu_1$ and $\mu_2$ are not necessarily different. In the subsequent result we exhibit a particular case  where this splitting in fact takes place.

\begin{cor}\label{c:1}
 Let us assume that (A1)--(A3) are satisfied and that $\lambda_N$ has multiplicity $2$ (i.e. $m=2$). Moreover, let us assume that $\omega$ is a disk  and that $\{u_N,u_{N+1}\}$ is an orthogonal basis of $E(\lambda_N)$. If, at the point $x=0$, the nodal lines of $ u_{N}$ are not tangent to any bisector between two nodal lines of $ u_{N+1}$, 
then $\mu_1 \neq \mu_2$ in Corollary \ref{cor:doubleEV}, item 2. Thus, the double eigenvalue $\lambda_{N}$ splits into two different branches $\lambda_{N}^\eps$ and $\lambda_{N+1}^\eps$. 
\end{cor}

The proof of the preceding corollary is contained in Section 5. It is deduced  from the general case of elliptic holes in dimension 2. Even in this general case, we find sufficient conditions for  eigenvalues splitting involving angles and coefficients appearing in the asymptotic behavior of eigenfunctions.

\section{Asymptotic behavior of $(u^a,u^b)$-capacities}

The aim of this section is to study the asymptotic behavior of $\mathrm{Cap}_\Omega(\varepsilon\overline{\omega},u^a,u^b)$ as $\varepsilon \to 0$. To reach  this goal, we consider the case of dimension $d=2$ and we assume some smoothness  on the sets and some regularity on the functions $u^a$ and $u^b$ (see assumptions (A1)-(A4)). We recall here these assumptions. We will work in the frame of Schauder classes: we take 
\[
\alpha\in]0,1[\,, 
\] 
 and we assume that
\begin{equation}\label{e1}
\begin{split}
&\text{$\Omega$ and $\omega$ are open bounded connected subsets of $\mathbb{R}^{2}$ of}\\
&\text{class $C^{1,\alpha}$ such that $\mathbb{R}^{{2}}\setminus\overline{\Omega}$  and $\mathbb{R}^{{2}}\setminus\overline{\omega}$ are connected,}\\
&\text{and  such that the origin  $0$ of $\mathbb{R}^{{2}}$ belongs both to $\Omega$ and $\omega$.}
\end{split}
\end{equation} 
 For the definition of functions and sets of the Schauder classes $C^{0,\alpha}$ and $C^{1,\alpha}$ we refer for example to Gilbarg and Trudinger~\cite[\S6.2]{GiTr01}.  Condition \eqref{e1} implies that there exists a real number $\varepsilon_0$ such that
\[
\varepsilon_0>0\ \mathrm{and\ }\ \varepsilon\overline{\omega} \subseteq \Omega\ \mathrm{for\ all}\ \varepsilon\in]-\varepsilon_0,\varepsilon_0[\,.
\]  
Then we set
\[
\Omega_\varepsilon\equiv\Omega\setminus K_\eps
\quad \text{where}\quad K_\eps\equiv \eps\overline\omega
\qquad\quad\forall\varepsilon\in]-\varepsilon_0,\varepsilon_0[\,.
\]
Clearly, $\Omega_\varepsilon$ is an open bounded connected subset of $\mathbb{R}^{2}$ of class $C^{1,\alpha}$ for all $\varepsilon\in]-\varepsilon_0,\varepsilon_{0}[\setminus\{0\}$. The boundary
$\partial \Omega_\varepsilon$ of $\Omega_\varepsilon$ is the union of the two connected components $\partial \Omega$ and $\partial (\varepsilon \omega)=\varepsilon\partial\omega$, for all
$\varepsilon\in]-\varepsilon_0,\varepsilon_{0}[$. We also note that $\Omega_0=\Omega\setminus\{0\}$. Moreover, $\{ \eps\overline\omega \}_{\eps>0}$ is a family of compact sets concentrating to the origin in the sense of \cite[Definition 1.2]{AbFeHiLe2019}.

As we have mentioned, we also need some regularity on the functions $u^a, u^b$: namely, we ask that
\begin{equation}\label{eq:assf}
\text{$u^a, u^b \in H^1(\Omega)$ are analytic in a neighborhood of $0$.}
\end{equation}
Our aim is twofold. On one side, we wish to obtain accurate and explicit expansions for $\mathrm{Cap}_\Omega(\varepsilon\overline{\omega},u^a,u^b)$ in terms of the parameter $\varepsilon$. On the other side, we also wish to emphasize the dependence on the geometric data of the problem (\i.e.,  $\Omega$ and $\omega$) and on the functions $u^a$ and $u^b$ on  $\mathrm{Cap}_\Omega(\varepsilon\overline{\omega},u^a,u^b)$. 

We observe that in this section we confine to the two-dimensional case. We exploit tools from potential theory and, as it happens often in this framework, the two-dimensional case and the one of dimension equal to or greater than three require a different analysis. This is mainly due to the different aspect of the fundamental solution of the Laplace equation: a logarithmic function of the $|x|$ if the dimension is two and a multiple of $|x|^{2-d}$ if the dimension $d$ is equal to or greater than three.

\subsection{Our strategy: the functional analytic approach}

In order to study the asymptotic behavior of  $\mathrm{Cap}_\Omega(\varepsilon\overline{\omega},u^a,u^b)$ as $\varepsilon \to 0$, we proceed as in \cite{AbBoLeMu} and we adopt the Functional Analytic Approach proposed by Lanza de Cristoforis \cite{La02, La04} for the analysis of singular perturbation problems in perforated domains (see also the monograph \cite{DaLaMu21} for a detailed presentation of the method). By applying this approach, one can deduce the possibility to represent the solution or related functionals as convergent power series. 

To analyze $\mathrm{Cap}_\Omega(\varepsilon\overline{\omega},u^a,u^b)$, we modify the results of \cite{AbBoLeMu}, where we considered the $u$-capacity $\mathrm{Cap}_\Omega(\varepsilon\overline{\omega},u)$. Although the modifications are quite straightforward, they require some attention to write out explicitly all the coefficients in the asymptotic expansions. For this reason, for the sake of clarity and completeness, we decided to include  in the present paper all the modified statements and some of the proofs.

By assumption \eqref{eq:assf}  on the analyticity of $u^a$ and $u^b$ together with analyticity results for the composition operator (see  B\"{o}hme and Tomi~\cite[p.~10]{BoTo73}, 
Henry~\cite[p.~29]{He82}, Valent~\cite[Thm.~5.2, p.~44]{Va88}), we deduce that, possibly shrinking $\varepsilon_0$, there exists two real analytic maps $U^a_\#$, $U^b_\#$ from $]-\varepsilon_0,\varepsilon_0[$ to $C^{1,\alpha}(\partial \omega)$ such that
\begin{equation}\label{eq:Usharp}
u^a(\varepsilon t)= U^a_\#[\varepsilon](t)\, ,\qquad u^b(\varepsilon t)= U^b_\#[\varepsilon](t) \, , \qquad \forall t \in \partial \omega\, ,\forall \varepsilon \in ]-\varepsilon_0,\varepsilon_0[
\end{equation}
(see Deimling~\cite[\S 15]{De85} for the 
definition and properties of analytic maps). Then for all $\varepsilon\in]-\varepsilon_{0},\varepsilon_{0}[ \setminus\{0\}$, we denote by $u^a_\varepsilon$ and $u^b_\varepsilon$  the unique solutions in $C^{1,\alpha}(\overline{\Omega_\varepsilon})$ of the  problems
\begin{equation}\label{eq:direpsa}
\left\{
\begin{array}{ll}
\Delta u^a_\varepsilon=0&\text{ in }\Omega_\varepsilon\,,\\
u^a_\varepsilon(x)=0&\text{ for all }x\in\partial\Omega\,,\\
u^a_\varepsilon(x)=U^a_\#[\varepsilon](x/\varepsilon)&\text{ for all }x\in\varepsilon\partial\omega\,
\end{array}
\right.
\end{equation}
and
\[%\begin{equation}\label{eq:direpsb}
\left\{
\begin{array}{ll}
\Delta u^b_\varepsilon=0&\text{ in }\Omega_\varepsilon\,,\\
u^b_\varepsilon(x)=0&\text{ for all }x\in\partial\Omega\,,\\
u^b_\varepsilon(x)=U^b_\#[\varepsilon](x/\varepsilon)&\text{ for all }x\in\varepsilon\partial\omega\,,
\end{array}
\right.
\]%\end{equation}
respectively. Clearly,
\[
\begin{split}
&V_{\varepsilon {\overline{\omega}}, u^a}(x)= u^a_\varepsilon(x)\, , \qquad V_{\varepsilon {\overline{\omega}}, u^b}(x)= u^b_\varepsilon(x)\, , \qquad \forall x \in \Omega_\varepsilon\, ,\forall \varepsilon \in ]-\varepsilon_0,\varepsilon_0[\setminus \{0\}\, ,\\
&V_{\varepsilon {\overline{\omega}}, u^a}(x)= u^a(x)\, ,\qquad V_{\varepsilon {\overline{\omega}}, u^b}(x)= u^b(x)\, , \qquad \forall x \in \varepsilon \overline{\omega}\, ,\forall \varepsilon \in ]-\varepsilon_0,\varepsilon_0[\setminus \{0\}\, {.}
\end{split}
\]
 Accordingly, by the Divergence Theorem, we have that
\begin{equation}\label{eq:div}
\begin{split}
\mathrm{Cap}_\Omega(\varepsilon\overline{\omega},u^a,u^b)&=\int_{\Omega_\varepsilon}\nabla u^a_\varepsilon \cdot \nabla u^b_\varepsilon \, dx + \int_{\varepsilon \omega}\nabla u^a \cdot \nabla u^b \, dx\\
&= -\int_{\partial (\varepsilon \omega)}\frac{\partial u^a_\varepsilon}{\partial \nu_{\varepsilon \omega}}u^b_\varepsilon \, d\sigma+ \varepsilon^2 \int_{\omega}(\nabla u^a)(\varepsilon t) \cdot (\nabla u^b)(\varepsilon t) \, dt\\
&= -\int_{\partial \omega}\nu_{\omega}(t)\cdot \nabla \Big(u^a_\varepsilon(\varepsilon t)\Big) u^b(\varepsilon t) \, d\sigma_t+ \varepsilon^2 \int_{\omega}(\nabla u^a)(\varepsilon t) \cdot (\nabla u^b)(\varepsilon t) \, dt \, ,
\end{split}
\end{equation}
for all $\varepsilon \in ]-\varepsilon_0,\varepsilon_0[\setminus \{0\}$. Here above the symbols $\nu_\omega$ and $\nu_{\varepsilon \omega}$ denote the outward unit normal to $\partial \omega$ and to $\partial (\varepsilon \omega)$, respectively.

As we have mentioned, our goal is to provide a fully constructive and complete asymptotic expansion for $\mathrm{Cap}_\Omega(\varepsilon\overline{\omega},u^a,u^b)$ as $\varepsilon \to 0$.  As  in \cite{AbBoLeMu}, to obtain an asymptotic expansion of $\mathrm{Cap}_\Omega(\varepsilon\overline{\omega},u^a,u^b)$, we follow the method developed in \cite{DaMuRo15} for the solution of the Dirichlet problem in a planar perforated domain.  
By the computation in \eqref{eq:div},  the quantity $\mathrm{Cap}_\Omega(\varepsilon\overline{\omega},u^a,u^b)$ can be expressed  as the sum of two integrals:
\[%\begin{equation}\label{eq:enuom}
\varepsilon^2 \int_{\omega}(\nabla u^a)(\varepsilon t) \cdot (\nabla u^b)(\varepsilon t) \, dt 
\]%\end{equation}
and the opposite of the integral on $\partial  \omega$ of the function  
\begin{equation}\label{eq:fun}
t \mapsto \nu_{\omega}(t)\cdot \nabla \Big(u^a_\varepsilon(\varepsilon t)\Big) u^b(\varepsilon t)\, .
\end{equation}

\subsection{Classical notions of potential theory}\label{prel1}

Formula \eqref{eq:div} shows that $\mathrm{Cap}_\Omega(\varepsilon\overline{\omega},u^a,u^b)$ can be expressed in terms of the gradient of the function $u^a_\varepsilon$ which solve problem \eqref{eq:direpsa} and of the traces of the (given) functions $u^a$ and $u^b$ As a consequence, we need to understand the behavior of the solution to problem \eqref{eq:direpsa} as $\varepsilon \to 0$ and, in order to do so, we shall exploit the approach of \cite{AbBoLeMu} based on integral operators, which allows to convert a boundary value problem into a set of integral equations defined on the boundary of the domain.

{ These operators are integral operators whose kernel is the fundamental solution of the Laplacian or its normal derivative. Therefore, we introduce    the fundamental solution $S$ {of $\Delta\equiv \sum_{j=1}^2\partial_{j}^2$} in $\mathbb{R}^2$ as the function from $\mathbb{R}^2\setminus\{0\}$ to $\mathbb{R}$ defined by
\[
S(x)\equiv\frac{1}{2\pi}\log\,|x|\qquad\forall x\in\mathbb{R}^2\setminus\{0\}\,.
\] 

Now let $\mathcal{O}$ be an open bounded subset of $\mathbb{R}^2$ of class $C^{1,\alpha}$.

We begin by introducing the single layer potential. If $\phi\in C^{0,\alpha}(\partial\mathcal{O})$,   then the single layer potential $v[\partial\mathcal{O},\phi]$ with density $\phi$ is the function defined by
\[
v[\partial\mathcal{O},\phi](x)\equiv\int_{\partial\mathcal{O}}\phi(y)S(x-y)\,d\sigma_y\qquad\forall x\in\mathbb{R}^2\\,
\] 
where $d\sigma$ denotes the arc length  element on $\partial\mathcal{O}$. The function $v[\partial\mathcal{O},\phi]$ is  continuous from $\mathbb{R}^2$ to $\mathbb{R}$. The restriction $v^+[\partial\mathcal{O},\phi]\equiv v[\partial\mathcal{O},\phi]_{|\overline{\mathcal{O}}}$ belongs to $C^{1,\alpha}(\overline{\mathcal{O}})$. Moreover, if  we denote by $C^{1,\alpha}_{\mathrm{loc}}(\mathbb{R}^2\setminus\mathcal{O})$  the space of functions on $\mathbb{R}^2\setminus\mathcal{O}$ whose restrictions to $\overline{\mathcal{U}}$ belong to  $C^{1,\alpha}(\overline{\mathcal{U}})$ for all open bounded subsets $\mathcal{U}$ of $\mathbb{R}^2\setminus\mathcal{O}$, then $v^-[\partial\mathcal{O},\phi]\equiv v[\partial\mathcal{O},\phi]_{|\mathbb{R}^2\setminus\mathcal{O}}$  belongs to $C^{1,\alpha}_{\mathrm{loc}}(\mathbb{R}^2\setminus\mathcal{O})$. 

Instead, for a function $\psi\in C^{1,\alpha}(\partial\mathcal{O})$, we denote  by $w[\partial\mathcal{O},\psi]$ the double layer potential with density $\psi$, namely
\[
w[\partial\mathcal{O},\psi](x)\equiv-\int_{\partial\mathcal{O}}\psi(y)\;\nu_{\mathcal{O}}(y)\cdot\nabla S(x-y)\,d\sigma_y\qquad\forall x\in\mathbb{R}^2\,,
\] 
where $\nu_\mathcal{O}$ denotes the outer unit normal to $\partial\mathcal{O}$  and the symbol $\cdot$ denotes the scalar product in $\mathbb{R}^2$. Then the restriction $w[\partial\mathcal{O},\psi]_{|\mathcal{O}}$ extends to a function $w^+[\partial\mathcal{O},\psi]$ of $C^{1,\alpha}(\overline{\mathcal{O}})$ and  the restriction $w[\partial\mathcal{O},\psi]_{|\mathbb{R}^2\setminus\overline{\mathcal{O}}}$ extends to a function $w^-[\partial\mathcal{O},\psi]$ of $C^{1,\alpha}_{\mathrm{loc}}(\mathbb{R}^2\setminus\mathcal{O})$.

In order to describe the properties of the trace of the double layer potential on $\partial \mathcal{O}$ and of the normal derivative of the single layer potential, we introduce the boundary integral operators  $W_\mathcal{O}$ and $W^*_\mathcal{O}$ as follows:
\[
W_\mathcal{O}[\psi](x)\equiv -\int_{\partial\mathcal{O}}\psi(y)\;\nu_{\mathcal{O}}(y)\cdot\nabla S(x-y)\, d\sigma_y\qquad\forall x\in\partial\mathcal{O}\,,
\] for all $\psi\in C^{1,\alpha}(\partial\mathcal{O})$, and
\[
W^*_\mathcal{O}[\phi](x)\equiv \int_{\partial\mathcal{O}}\phi(y)\;\nu_{\mathcal{O}}(x)\cdot\nabla S(x-y)\, d\sigma_y\qquad\forall x\in\partial\mathcal{O}\,,
\] for all $\phi\in C^{0,\alpha}(\partial\mathcal{O})$. Then $W_\mathcal{O}$ is a compact operator  from   $C^{1,\alpha}(\partial\mathcal{O})$  to itself and $W^*_\mathcal{O}$   is a compact operator  from  $C^{0,\alpha}(\partial\mathcal{O})$ to itself (see Schauder \cite{Sc31} and \cite{Sc32}). The operators $W_\mathcal{O}$ and $W^*_\mathcal{O}$ are adjoint one to the other with respect to the duality on $C^{1,\alpha}(\partial\mathcal{O})\times C^{0,\alpha}(\partial\mathcal{O})$ induced by the inner product of the Lebesgue space $L^2(\partial\mathcal{O})$ (see ,  {\it e.g.}, Kress \cite[Chap.~4]{Kr14}). For the theory of dual systems and the corresponding Fredholm Alternative Principle, we refer the reader to Kress \cite{Kr14} and Wendland \cite{We67,We70}. By means of the operators $W_\mathcal{O}$ and $W^*_\mathcal{O}$, we can describe the traces $w^\pm[\partial\mathcal{O},\psi]_{|\partial\mathcal{O}}$ and the normal derivatives $\nu_\mathcal{O}\cdot\nabla v^\pm[\partial\mathcal{O},\phi]_{|\partial\mathcal{O}}$:
\begin{align*}
w^\pm[\partial\mathcal{O},\psi]_{|\partial\mathcal{O}}&=\pm\frac{1}{2}\psi+W_\mathcal{O}[\psi]&\forall\psi\in C^{1,\alpha}(\partial\mathcal{O})\,,\\
\nu_\mathcal{O}\cdot\nabla v^\pm[\partial\mathcal{O},\phi]_{|\partial\mathcal{O}}&=\mp\frac{1}{2}\phi+W^*_\mathcal{O}[\phi]&\forall\phi\in C^{0,\alpha}(\partial\mathcal{O})
\end{align*}
(see, {\it e.g.}, Folland \cite[Chap.~3]{Fo95}).

We shall need to consider subspaces of $C^{0,\alpha}(\partial \mathcal{O})$ and of $C^{1,\alpha}(\partial \mathcal{O})$, consisting of functions with zero integral on $\partial \mathcal{O}$. Therefore,  we set
\[
C^{k,\alpha}(\partial \mathcal{O})_0\equiv \Bigg\{f \in C^{k,\alpha}(\partial \mathcal{O})\colon \int_{\partial \mathcal{O}}f\, d\sigma=0\Bigg\} \qquad \text{for $k=0,1$}\, .
\] 
}

\subsection{An integral formulation of the boundary value problem}

To convert problem \eqref{eq:direpsa} into a system of integral equations, we follow the strategy of Lanza de Cristoforis \cite{La08} and of \cite{DaMuRo15}. As in \cite{AbBoLeMu}, we divide the problem in a part which can be solved in terms of the double layer potential and a part which will be represented by a single layer potential. 
Now we proceed as in \cite{DaMuRo15} and we introduce the map $M\equiv(M^o,M^i,M^c)$  from $]-\varepsilon_0,\varepsilon_0[\times C^{0,\alpha}(\partial\Omega)\times C^{0,\alpha}(\partial\omega)$ to $C^{0,\alpha}(\partial\Omega)\times C^{0,\alpha}(\partial\omega)_0\times\mathbb{R}$   by setting
\begin{align*}
&M^o[\varepsilon,\rho^o,\rho^i](x) \equiv \frac{1}{2}\rho^o(x)+W^*_{\Omega}[\rho^o](x)+\int_{\partial\omega}\rho^i(s)\;\nu_{\Omega}(x)\cdot\nabla S(x-\varepsilon s)\, d\sigma_s&\forall x\in\partial\Omega\,,\\
&M^i[\varepsilon,\rho^o,\rho^i](t)\equiv\frac{1}{2}\rho^i(t)-W^*_{\omega}[\rho^i](t)-\varepsilon \int_{\partial\Omega}\rho^o(y)\;\nu_{\omega}(t)\cdot\nabla S(\varepsilon t- y)\, d\sigma_y&\forall t\in\partial\omega\,,\\
&M^c[\varepsilon,\rho^o,\rho^i]\equiv \int_{\partial\omega}\rho^i\, d\sigma-1\,,
\end{align*} 
for all $(\varepsilon,\rho^o,\rho^i) \in ]-\varepsilon_0,\varepsilon_0[\times C^{0,\alpha}(\partial\Omega)\times C^{0,\alpha}(\partial\omega)$.
Then we  state the following result of Lanza de Cristoforis \cite[\S3]{La08} (see also \cite[Prop.~4.1]{DaMuRo15}).

\begin{prop}\label{rhoeps}
The following statements hold.
\begin{itemize}
\item[(i)] The map $M$ is real analytic. 
\item[(ii)] If $\varepsilon\in]-\varepsilon_0,\varepsilon_0[$, then there exists a unique pair $(\rho^o[\varepsilon],\rho^i[\varepsilon])\in C^{0,\alpha}(\partial\Omega)\times C^{0,\alpha}(\partial\omega)$ such that $M[\varepsilon,\rho^o[\varepsilon],\rho^i[\varepsilon]]=0$. \item[(iii)] The map from $]-\varepsilon_0,\varepsilon_0[$ to $C^{0,\alpha}(\partial\Omega)\times C^{0,\alpha}(\partial\omega)$ which takes $\varepsilon$ to   $(\rho^o[\varepsilon],\rho^i[\varepsilon])$ is real analytic.
\end{itemize}
\end{prop}

\begin{rem}\label{rem:rhoeps}
For each $\varepsilon \in ]-\varepsilon_0,\varepsilon_0[\setminus \{0\}$, let $\tau_\varepsilon$ be defined by $\tau_\varepsilon(x)\equiv\rho^o[\varepsilon](x)$ for all $x \in \partial \Omega$ and $\tau_\varepsilon(x)\equiv |\varepsilon|^{-1}\rho^i[\varepsilon](x/\varepsilon)$ for all $x \in \partial (\varepsilon \omega)$. Then
\[
\frac{1}{2}\tau_\varepsilon+W^*_{\Omega_\varepsilon}[\tau_\varepsilon]=0 \, , \qquad \int_{\partial (\varepsilon \omega)}\tau_\varepsilon\, d\sigma=1\, ,
\]
for all $\varepsilon \in ]-\varepsilon_0,\varepsilon_0[\setminus \{0\}$.
\end{rem}

We now consider the part which can be solved by the double layer potential.  For  $\varepsilon \in ]-\varepsilon_0,\varepsilon_0[\setminus \{0\}$, we define the boundary datum $g^a_\varepsilon$ by setting
\[
g^a_\varepsilon(x)\equiv 0 \quad \forall x \in \partial \Omega\, , \qquad g^a_\varepsilon(x)=U^a_\#[\varepsilon](x/\varepsilon)-\int_{\partial (\varepsilon \omega)}U^a_\#[\varepsilon](y/\varepsilon)\tau_\varepsilon(y)\, d\sigma_y \quad \forall x \in \partial (\varepsilon \omega)\, ,
\]
Standard Fredholm theory and classical potential theory imply that for $\varepsilon \in ]-\varepsilon_0,\varepsilon_0[\setminus \{0\}$ the function $g^a_\varepsilon$ belongs to the image of the trace of the double layer potential (for the definition of $U^a_\#$ see \eqref{eq:Usharp}). As in \cite{DaMuRo15}, we then define the map
$\Lambda\equiv(\Lambda^o,\Lambda^i)$  from $]-\varepsilon_0,\varepsilon_0[\times  C^{1,\alpha}(\partial\Omega)\times C^{1,\alpha}(\partial\omega)_0$ to $C^{1,\alpha}(\partial\Omega)\times C^{1,\alpha}(\partial\omega)$  by
\begin{align*}
&\Lambda^o[\varepsilon,\theta^o,\theta^i](x)\equiv\frac{1}{2}\theta^o(x)+W_{\Omega}[\theta^o](x) \\
&\qquad\qquad\qquad\qquad +\varepsilon\int_{\partial\omega}\theta^i(s)\;\nu_{\omega}(s)\cdot\nabla S(x-\varepsilon s)\, d\sigma_s&\forall x\in\partial\Omega\,,\\
&\Lambda^i[\varepsilon,\theta^o,\theta^i](t)\equiv\frac{1}{2}\theta^i(t)-W_{\omega}[\theta^i](t)+w[\partial\Omega,\theta^o](\varepsilon t)\\
\nonumber
&\qquad\qquad\qquad\qquad -U^a_\#[\varepsilon](t)+\int_{\partial\omega}U^a_\#[\varepsilon]\rho^i[\varepsilon]\,d\sigma&\forall t\in\partial\omega\, ,
\end{align*} 
for all $(\varepsilon,\theta^o,\theta^i) \in ]-\varepsilon_0,\varepsilon_0[\times  C^{1,\alpha}(\partial\Omega)\times C^{1,\alpha}(\partial\omega)_0$. Then we have the following result of Lanza de Cristoforis \cite[\S4]{La08} on the regularity of $\Lambda$ (see  \cite[Prop.~4.3]{DaMuRo15}).

\begin{prop}\label{thetaeps} 
The following statements hold.
\begin{itemize}
\item[(i)] The map $\Lambda$ is real analytic.
\item[(ii)] If $\varepsilon\in ]-\varepsilon_0,\varepsilon_0[$, then there exists a unique pair $(\theta^o[\varepsilon],\theta^i[\varepsilon])\in C^{1,\alpha}(\partial\Omega)\times C^{1,\alpha}(\partial\omega)_0$ such that $\Lambda[\varepsilon,\theta^o[\varepsilon],\theta^i[\varepsilon]]=0$.
\item[(iii)]  The map from $]-\varepsilon_0,\varepsilon_0[$ to $C^{1,\alpha}(\partial\Omega)\times C^{1,\alpha}(\partial\omega)_0$ which takes $\varepsilon$ to $(\theta^o[\varepsilon], \theta^i[\varepsilon])$ is real analytic.
\end{itemize}
\end{prop}

\begin{rem}\label{rem:thetaeps}
For each $\varepsilon \in ]-\varepsilon_0,\varepsilon_0[\setminus \{0\}$, let $\mu_\varepsilon$ be defined by $\mu_\varepsilon(x)\equiv\theta^o[\varepsilon](x)$ for all $x \in \partial \Omega$ and $\mu_\varepsilon(x)\equiv \theta^i[\varepsilon](x/\varepsilon)$ for all $x \in \partial (\varepsilon \omega)$. Then
\[
\frac{1}{2}\mu_\varepsilon+W_{\Omega_\varepsilon}[\mu_\varepsilon]=g^a_\varepsilon \, ,
\]
for all $\varepsilon \in ]-\varepsilon_0,\varepsilon_0[\setminus \{0\}$.
\end{rem}

We can recover the solution $u^a_\varepsilon$ (and in particular the rescaled function $t \mapsto u^a_\varepsilon(\varepsilon t)$) by summing the double layer potential with density $\mu_\varepsilon$ (see  Remark \ref{rem:thetaeps}) and a convenient multiple of the single layer potential with density $\tau_\varepsilon$ (see  Remark \ref{rem:rhoeps}). By arguing as in \cite[Prop.~4.5]{DaMuRo15}, we show in the following Proposition \ref{solution}   how to represent  the rescaled function $u^a_\varepsilon(\varepsilon t)$ by means of the functions $\rho^o[\varepsilon]$, $\rho^i[\varepsilon]$, $\theta^o[\varepsilon]$, and $\theta^i[\varepsilon]$ introduced in Propositions \ref{rhoeps} and \ref{thetaeps} (see also Lanza de Cristoforis \cite[\S5]{La08} and \cite[\S 2.4]{DaMu13}).

\begin{prop}\label{solution}
Let $\varepsilon\in]-\varepsilon_0,\varepsilon_0[\setminus\{0\}$. Then
\[
\begin{split}
u^a_\varepsilon(\varepsilon t)& \equiv w^+[\partial\Omega,\theta^o[\varepsilon]](\varepsilon t)-w^-[\partial\omega,\theta^i[\varepsilon]](t)\\
&
+\int_{\partial\omega}U^a_\#[\varepsilon]\rho^i[\varepsilon]\,d\sigma \biggl(v^+[\partial\Omega,\rho^o[\varepsilon]](\varepsilon t) +v^-[\partial\omega,\rho^i[\varepsilon]](t)+\frac{\log |\varepsilon|}{2\pi}\biggr)\\
&
\times\biggl(\frac{1}{\int_{\partial\omega}d\sigma}\int_{\partial\omega}v[\partial\Omega,\rho^o[\varepsilon]](\varepsilon s)+v[\partial\omega,\rho^i[\varepsilon]](s)\,d\sigma_s+\frac{\log |\varepsilon|}{2\pi}\biggr)^{-1}
\end{split}
\] 
for all $t\in\overline{(\varepsilon^{-1}\Omega)}\setminus\omega$.
\end{prop}

\subsection{Power series expansions of the auxiliary functions  $(\rho^o[\varepsilon],\rho^i[\varepsilon])$ and $(\theta^o[\varepsilon],\theta^i[\varepsilon])$ around $\varepsilon=0$}\label{sec3}

As in \cite{AbBoLeMu}, the plan is first to construct an expansion for $\nu_{\omega}(t)\cdot \nabla \Big(u^a_\varepsilon(\varepsilon t)\Big) u^b(\varepsilon t)$ and then to integrate such an expansion on $\partial \omega$. We note that $u^a_\varepsilon(\varepsilon t)$ is represented by means of the auxiliary {density} functions  $(\rho^o[\varepsilon],\rho^i[\varepsilon])$ and $(\theta^o[\varepsilon],\theta^i[\varepsilon])$. Thus the plan is to obtain an expansion for those densities and then to get the one for $u^a_\varepsilon(\varepsilon t)$ by exploiting the representation formula of Proposition \ref{solution}.

The following Proposition \ref{rhok} of \cite[Prop.~5.1]{DaMuRo15} provides a power series expansion around $0$ of $(\rho^o[\varepsilon],\rho^i[\varepsilon])$. Throughout the paper, if $j\in\{1,2\}$, then $(\partial_j F)(y)$ denotes the partial derivative with respect to $x_j$ of the function $F(x)\equiv F(x_1,x_2)$ evaluated at $y\equiv(y_1,y_2)\in\mathbb{R}^2$.

\begin{prop}\label{rhok} Let $(\rho^o[\varepsilon],\rho^i[\varepsilon])$ be as in Proposition \ref{rhoeps} for all $\varepsilon\in]-\varepsilon_0,\varepsilon_0[$. Then there exist $\varepsilon_\rho\in]0,\varepsilon_0[$ and a sequence $\{(\rho^o_k,\rho^i_k)\}_{k\in\mathbb{N}}$ in $C^{0,\alpha}(\partial\Omega)\times C^{0,\alpha}(\partial\omega)$ such that 
\[%\begin{equation}\label{rhok0}
\rho^o[\varepsilon]=\sum_{k=0}^{+\infty}\frac{\rho^o_k}{k!}\varepsilon^k\quad\text{ and }\quad\rho^i[\varepsilon]=\sum_{k=0}^{+\infty}\frac{\rho^i_k}{k!}\varepsilon^k\qquad \forall \varepsilon\in]-\varepsilon_\rho,\varepsilon_\rho[\,,
\]%\end{equation}
where the two series converge normally in $C^{0,\alpha}(\partial\Omega)$ and in $C^{0,\alpha}(\partial\omega)$, respectively, for all $\varepsilon\in]-\varepsilon_\rho,\varepsilon_\rho[$. Moreover,  the pair of functions $(\rho^o_0,\rho^i_0)$ is the unique solution in $C^{0,\alpha}(\partial\Omega)
\times C^{0,\alpha}(\partial\omega)$ of the following system of integral equations
\begin{align}\nonumber%\label{rhoo0.0}
&\frac{1}{2}\rho^o_0(x)+W^*_{\Omega}[\rho^o_0](x)=-\nu_{\Omega}(x)\cdot\nabla S(x) & \forall x\in\partial\Omega\,,\\
\nonumber%\label{rhoi0.1}
&\frac{1}{2} \rho^i_0(t)-W^*_{\omega}[ \rho^i_0](t)=0 & \forall t\in\partial\omega\,,\\
\nonumber%\label{rhoi0.2}
&\int_{\partial\omega}\rho^i_0\, d\sigma=1\,, 
\end{align}
and for each $k\in\mathbb{N}\setminus\{0\}$ the pair $(\rho^o_k,\rho^i_k)$ is the unique solution in $C^{0,\alpha}(\partial\Omega)
\times C^{0,\alpha}(\partial\omega)$ of the following system of integral equations which involves $\{(\rho^o_j,\rho^i_j)\}_{j=0}^{k-1}$,
\begin{align}\nonumber%\label{rhook}
&\frac{1}{2}\rho^o_k(x)+W^*_{\Omega}[\rho^o_k](x)\\
\nonumber
&\quad =\sum_{j=0}^{k}\binom{k}{j}(-1)^{j+1}\sum_{h=0}^j\binom{j}{h}\nu_{\Omega}(x)\cdot(\nabla\partial_1^h\partial_2^{j-h} S)(x)\int_{\partial\omega}\rho^i_{k-j}(s)s_1^hs_2^{j-h}\, d\sigma_s& \forall x\in\partial\Omega\,,\\
\nonumber%\label{rhoik1}
&\frac{1}{2} \rho^i_k(t)-W^*_{\omega}[ \rho^i_k](t)\\
\nonumber
&\quad =k\sum_{j=0}^{k-1}\binom{k-1}{j}(-1)^{j+1}\sum_{h=0}^j\binom{j}{h}t_1^h t_2^{j-h}\nu_{\omega}(t)\cdot\int_{\partial\Omega}\rho^o_{k-1-j}(\nabla \partial_1^h\partial_2^{j-h} S)\, d\sigma&  \forall t\in\partial\omega\,,\\
\nonumber%\label{rhoik2}
&\int_{\partial\omega}\rho^i_k\, d\sigma=0\,.
\end{align}  
\end{prop}

The coefficients in the power series expansion of $(\theta^o[\varepsilon],\theta^i[\varepsilon])$ are instead determine in the following Proposition \ref{thetak} (see  \cite[Prop.~3.2]{AbBoLeMu}).

\begin{prop}\label{thetak} Let $(\theta^o[\varepsilon],\theta^i[\varepsilon])$ be as in Proposition \ref{thetaeps} for all $\varepsilon\in]-\varepsilon_0,\varepsilon_0[$. Then there exist $\varepsilon_\theta\in]0,\varepsilon_0[$ and a sequence $\{(\theta^o_k,\theta^i_k)\}_{k\in\mathbb{N}}$ in $C^{1,\alpha}(\partial\Omega)\times C^{1,\alpha}(\partial\omega)_0$ such that 
\[%\begin{equation}\label{thetak0}
\theta^o[\varepsilon]=\sum_{k=0}^\infty\frac{\theta^o_k}{k!}\varepsilon^k\quad\text{and }\quad\theta^i[\varepsilon]=\sum_{k=0}^\infty\frac{\theta^i_k}{k!}\varepsilon^k\qquad\forall \varepsilon\in]-\varepsilon_\theta,\varepsilon_\theta[\,,
\]%\end{equation} 
\[
(\theta^o_0,\theta^i_0)=(0,0)\, , \qquad \theta^o_1=0\, ,
\]
and $\theta^i_1$ is the unique solution in $C^{1,\alpha}(\partial\omega)_0$ of
\[%\begin{equation}\label{thetai1}
\begin{split}
\frac{1}{2}&\theta^i_1(t)-W_{\omega}[\theta^i_1](t)\\&=\sum_{h=0}^1t_1^h t_2^{1-h} (\partial_1^h\partial_2^{1-h}u^a)(0)-\sum_{l=0}^1 \sum_{h=0}^l \int_{\partial\omega}s_1^hs_2^{l-h}(\partial_1^h \partial_2^{l-h}u^a)(0)\rho^i_{1-l}(s)\,d\sigma_s\quad \forall t\in\partial\omega\,{,}
\end{split}
\]%\end{equation}
and for each $k\in\mathbb{N}\setminus\{0,1\}$ the pair $(\theta^o_k,\theta^i_k)$ is the unique solution in $C^{1,\alpha}(\partial\Omega)\times C^{1,\alpha}(\partial\omega)_0$ of the following  system of integral equations which involves $\{(\theta^o_j,\theta^i_j)\}_{j=0}^{k-1}$,
\begin{align}\nonumber%\label{thetaok}
&\frac{1}{2}\theta^o_k(x)+W_{\Omega}[\theta^o_k](x)\\
\nonumber
&=k\sum_{j=0}^{k-2}\binom{k-1}{j}(-1)^{j+1}\sum_{h=0}^j\binom{j}{h}(\nabla \partial_1^h \partial_2^{j-h} S)(x)\cdot \int_{\partial\omega}\theta^i_{k-1-j}(s)\;\nu_{\omega}(s)s_1^h s_2^{j-h} d\sigma_s\\ 
\nonumber
&\qquad\qquad\qquad\qquad\qquad\qquad\qquad\qquad\qquad\qquad\qquad\qquad\qquad\qquad\qquad \forall x\in\partial\Omega\,,\\
\nonumber%\label{thetaik}
&\frac{1}{2}\theta^i_k(t)-W_{\omega}[\theta^i_k](t)=\sum_{j=0}^{k-1}\binom{k}{j}(-1)^{j+1}\sum_{h=0}^j\binom{j}{h}t_1^h t_2^{j-h}\int_{\partial\Omega}\theta^o_{k-j}\nu_{\Omega}\cdot \nabla \partial_1^h \partial_2^{j-h} S \,d\sigma\\ \nonumber
&+\sum_{h=0}^k\binom{k}{h}t_1^h t_2^{k-h} (\partial_1^h\partial_2^{k-h}u^a)(0)-\sum_{l=0}^k \sum_{h=0}^l \binom{k}{l}\binom{l}{h}\int_{\partial\omega}s_1^hs_2^{l-h}(\partial_1^h \partial_2^{l-h}u^a)(0)\rho^i_{k-l}(s)\,d\sigma_s\\ \nonumber
&\qquad\qquad\qquad\qquad\qquad\qquad\qquad\qquad\qquad\qquad\qquad\qquad\qquad\qquad\qquad  \forall t\in\partial\omega\,.\nonumber
\end{align}
\end{prop}

\subsection{Series expansion of $\nu_{\omega}(\cdot)\cdot \nabla \big(u^a_\varepsilon(\varepsilon \cdot)\big) u^b(\varepsilon \cdot)$ around $\varepsilon=0$}\label{sec4}

{
In order to compute an asymptotic expansion of
\[
-\int_{\partial \omega}\nu_{\omega}(t)\cdot \nabla \Big(u^a_\varepsilon(\varepsilon t)\Big) u^b(\varepsilon t) \, d\sigma_t
\]
as $\varepsilon \to 0$, we now turn to construct a series expansion for $\nu_{\omega}(\cdot)\cdot \nabla \big(u^a_\varepsilon(\varepsilon \cdot)\big) u^b(\varepsilon \cdot)$ for $\varepsilon$ in a neighborhood of $0$. The coefficients of such expansion will be defined by means of the derivatives of the functions $u^a$ and $u^b$ and of the sequences $\{(\rho^o_k,\rho^i_k)\}_{k\in\mathbb{N}}$ and $\{(\theta^o_k,\theta^i_k)\}_{k\in\mathbb{N}}$  introduced in Section \ref{sec3}. 
As in \cite[Prop.~6.1]{DaMuRo15}, in the following Proposition \ref{uk}, we prove a representation formula which can be easily obtained by   Propositions \ref{solution}, \ref{rhok} and \ref{thetak}, and by   standard properties of real analytic maps (see also  Lanza de Cristoforis \cite[Theorem 5.3]{La08} and \cite[Theorem 3.1]{DaMu13}).}

\begin{prop}\label{uk} Let $\{(\rho^o_k,\rho^i_k)\}_{k\in\mathbb{N}}$ and $\{(\theta^o_k,\theta^i_k)\}_{k\in\mathbb{N}}$ be as in Propositions \ref{rhok} and \ref{thetak}, respectively.  Let 
\[
\begin{split}
&u^a_{\mathrm{m},0}(t)\equiv0\qquad\qquad\qquad\qquad\qquad\qquad\qquad\qquad  \forall t\in\mathbb{R}^2\setminus\omega\,, \\
& u^a_{\mathrm{m},1}(t)\equiv- w^-[\partial\omega,\theta^i_1](t) \qquad\qquad\qquad\qquad\qquad  \forall t\in\mathbb{R}^2\setminus\omega\,, \\
&u^a_{\mathrm{m},k}(t)\equiv\frac{1}{k!}\sum_{j=0}^{k-1}\binom{k}{j}(-1)^{j}\sum_{h=0}^j\binom{j}{h} t_1^h t_2^{j-h} \int_{\partial\Omega}\theta^o_{k-j}\,\nu_{\Omega}\cdot(\nabla\partial_1^h \partial_2^{j-h} S)\,d\sigma\\
&\qquad\qquad - \frac{1}{k!}w^-[\partial\omega,\theta^i_k](t) \quad\qquad\qquad\qquad\qquad  \forall t\in\mathbb{R}^2\setminus\omega\,, \quad \forall k \geq 2 \\
\end{split}
\]
and
 \[
\begin{split}
&v_{\mathrm{m},k}(t)\equiv\frac{1}{k!}\sum_{j=0}^k \binom{k}{j}(-1)^j\sum_{h=0}^j\binom{j}{h} t_1^h t_2^{j-h}\int_{\partial\Omega}\rho^o_{k-j}\partial_1^h \partial_2^{j-h} S\,d\sigma
 +\frac{1}{k!}v^-[\partial\omega,\rho^i_k](t)\\
&\qquad\qquad\qquad\qquad\qquad\qquad\qquad\qquad\qquad\qquad\qquad\qquad\qquad\qquad \forall t\in\mathbb{R}^2\setminus\omega\,, \\
&g^a_k\equiv\frac{1}{k!}\sum_{l=0}^k \sum_{h=0}^l \binom{k}{l}\binom{l}{h}\int_{\partial\omega}s_1^hs_2^{l-h}(\partial_1^h \partial_2^{l-h}u^a)(0)\rho^i_{k-l}(s)\,d\sigma_s\,,\\
&r_k\equiv\frac{1}{k!\int_{\partial\omega}d\sigma}\sum_{j=0}^k\binom{k}{j}(-1)^j\sum_{h=0}^j\binom{j}{h} \int_{\partial\omega}s_1^h s_2^{j-h}\,d\sigma_s\int_{\partial\Omega}\rho^o_{k-j}\partial_1^h \partial_2^{j-h}S\,d\sigma\\
&\qquad+\frac{1}{k!\int_{\partial\omega}d\sigma}\int_{\partial\omega}v[\partial\omega,\rho^i_k]\,d\sigma\,,
\end{split}
\] for all $k\in\mathbb{N}$. Then  the following statements hold.
\begin{enumerate}
\item[(i)] There exists $\varepsilon^*\in]0,\varepsilon_0]$ such that the series $\sum_{k=0}^\infty g^a_{k}\varepsilon^k$ and $\sum_{k=0}^\infty r_k\varepsilon^k$ converge absolutely for all $\varepsilon \in ]-\varepsilon^*,\varepsilon^*[$. Moreover,
\[
g^a_0=u^a(0)\, .
\]
\item[(ii)] If $\Omega_\mathrm{m}\subseteq\mathbb{R}^2\setminus\overline{\omega}$ is open and bounded,   then there exists $\varepsilon_\mathrm{m}\in]0,\varepsilon^*]\cap]0,1[$  such that $\varepsilon\overline{\Omega}_\mathrm{m}\subseteq\Omega$  for all $\varepsilon\in]-\varepsilon_\mathrm{m},\varepsilon_\mathrm{m}[$ and such that
\begin{equation}\label{uepsm}
u^a_{\varepsilon}(\varepsilon\cdot)_{|\overline{\Omega}_\mathrm{m}}=\sum_{k=1}^\infty u^a_{\mathrm{m},k|\overline{\Omega}_\mathrm{m}}\varepsilon^k+(\sum_{k=0}^\infty g^a_{k}\varepsilon^k)\frac{\sum_{k=0}^\infty v_{\mathrm{m},k|\overline{\Omega}_\mathrm{m}}\varepsilon^k+(2\pi)^{-1}\log|\varepsilon|}{\sum_{k=0}^\infty r_k\varepsilon^k+(2\pi)^{-1}\log|\varepsilon|}
\end{equation} for all $\varepsilon\in]-\varepsilon_\mathrm{m},\varepsilon_\mathrm{m}[\setminus\{0\}$. Moreover, the series $\sum_{k=1}^\infty u^a_{\mathrm{m},k|\overline{\Omega}_\mathrm{m}}\varepsilon^k$ and $\sum_{k=0}^\infty v_{\mathrm{m},k|\overline{\Omega}_\mathrm{m}}\varepsilon^k$ converge  normally in $C^{1,\alpha}(\overline{\Omega}_\mathrm{m})$  for all $\varepsilon\in]-\varepsilon_\mathrm{m},\varepsilon_\mathrm{m}[$.
\end{enumerate}
\end{prop}

We are now ready to obtain an expansion for the map in \eqref{eq:fun} by exploiting Proposition \ref{uk}.

\begin{prop}\label{funk} 
With the notation introduced in Proposition \ref{uk}, let 
\[
\begin{split}
&u^l_{\#,k}(t)\equiv\sum_{\substack{(h,j)\in \mathbb{N}^2\\ h+j=k}}\frac{\partial_1^h \partial_2^{j}u^l(0)}{h! j!}t^h_1 t_2^j\qquad\qquad\quad  \forall t\in  \mathbb{R}^2\,, \quad l=a,b\, , \\
& \tilde{u}_{k}(t)\equiv\sum_{l=0}^k \nu_{\omega}(t)\cdot\nabla u^a_{\mathrm{m},l |\partial \omega}(t) u^b_{\#,k-l}(t)
\qquad  \forall t\in\partial\omega\,, \\
& \tilde{v}_k(t)\equiv \nu_{\omega}(t)\cdot\nabla v_{\mathrm{m},k|\partial \omega}(t) \quad \ \qquad \qquad \qquad  \forall t\in\partial\omega\,, \\
&\tilde{g}_{k}(t)\equiv\sum_{l=0}^k g^a_l u^b_{\#,k-l}(t) \quad\qquad\qquad\qquad\qquad  \forall t\in\partial \omega\,, \end{split}
\]
for all $k\in\mathbb{N}$. Then  there exists $\tilde{\varepsilon}\in]0,\varepsilon^{*}]\cap]0,1[$  such that
\begin{equation}\label{funepsm}
\begin{split}
\nu_{\omega}(t) \cdot \nabla \big(u^a_{\varepsilon}(\varepsilon t)\big) u^b(\varepsilon t) =\sum_{k=1}^\infty \tilde{u}_{k}(t)\varepsilon^k+\Bigg(\sum_{k=0}^\infty \tilde{g}_{k}(t)\varepsilon^k\Bigg)\frac{\sum_{k=0}^\infty \tilde{v}_{k}(t)\varepsilon^k}{\sum_{k=0}^\infty r_k\varepsilon^k+(2\pi)^{-1}\log|\varepsilon|}
\quad \forall t \in \partial \omega\, , \end{split}
\end{equation} for all $\varepsilon\in]-\tilde{\varepsilon},\tilde{\varepsilon}[\setminus\{0\}$. Moreover, the series $\sum_{k=0}^\infty \tilde{g}_{k}\varepsilon^k$, $\sum_{k=0}^\infty \tilde{u}_{k}\varepsilon^k$, and $\sum_{k=0}^\infty \tilde{v}_{k}\varepsilon^k$ converge  normally in $C^{0,\alpha}(\partial \omega)$  for all $\varepsilon\in]-\tilde{\varepsilon},\tilde{\varepsilon}[$.
\end{prop}
\proof
Taking  $\tilde{\varepsilon}\in ]0,\varepsilon^{\ast}[$ small enough,  we have that 
\[
\begin{split}
u^l(\varepsilon t)&=\sum_{(i,j)\in \mathbb{N}^2}\varepsilon^{i+j}\frac{\partial_1^i \partial_2^{j}u^l(0)}{i! j!}t^i_1 t_2^j\\
&=\sum_{h=0}^\infty \Bigg(\sum_{\substack{(i,j)\in \mathbb{N}^2\\ i+j=h}}\frac{\partial_1^i \partial_2^{j}u^l(0)}{i! j!}t^i_1 t_2^j\Bigg) \varepsilon^{h}=\sum_{h=0}^\infty u^l_{\#,h}(t) \varepsilon^{h} \qquad \forall t \in \partial \omega\, , \quad l=a,b\, ,
\end{split}
\]
for $\varepsilon \in ]-\tilde{\varepsilon},\tilde{\varepsilon}[$, and that the power series $\sum_{h=0}^\infty u^l_{\#,h|\partial \omega} \varepsilon^{h} $ converges normally in $C^{0,\alpha}(\partial \omega)$  for all $\varepsilon \in ]-\tilde{\varepsilon},\tilde{\varepsilon}[$, for $l=a,b$.
We observe that,  possibly taking a smaller $\tilde{\varepsilon}$, 
\[
\Big(\sum_{k=1}^\infty \nu_{\omega} \cdot \nabla u^a_{\mathrm{m},k|\partial \omega}\varepsilon^k\Big)\Big(\sum_{h=0}^\infty u^b_{\#,h|\partial \omega} \varepsilon^{h}\Big)=\sum_{k=0}^\infty \tilde{u}_{k}\varepsilon^k\, ,  \qquad \Big(\sum_{k=0}^\infty g^a_{k}\varepsilon^k\Big)\Big(\sum_{h=0}^\infty u^b_{\#,h|\partial \omega} \varepsilon^{h}\Big)=\sum_{k=0}^\infty \tilde{g}_{k}\varepsilon^k
\]
where the series converge normally in $C^{0,\alpha}(\partial \omega)$  for all $\varepsilon \in ]-\tilde{\varepsilon},\tilde{\varepsilon}[$ and we have set
\[
\tilde{u}_{k}\equiv\sum_{l=0}^k \nu_{\omega} \cdot \nabla u^a_{\mathrm{m},l |\partial \omega} u^b_{\#,k-l|\partial \omega}\, , \qquad \tilde{g}_{k}\equiv\sum_{l=0}^k g^a_l u^b_{\#,k-l|\partial \omega}\, . 
\]
Then by Proposition \ref{uk} (see formula \eqref{uepsm}), we deduce the validity of \eqref{funepsm}. \qed

The next step is to represent $\nu_{\omega}(\cdot) \cdot \nabla \big(u^a_{\varepsilon}(\varepsilon\cdot)\big)_{|\partial \omega}u^b(\varepsilon \cdot)_{|\partial \omega}$ as a convergent series of the type
\[
\sum_{n=0}^\infty  \varphi_\varepsilon(\cdot) \varepsilon^n \, .
\]
By exploiting exactly the same argument of \cite[Thm.~6.3]{DaMuRo15}, we can prove  Theorem  \ref{umk} below where we obtain  a series expansion for the map
\[
\varepsilon \mapsto \nu_{\omega}(\cdot) \cdot \nabla \big(u^a_{\varepsilon}(\varepsilon\cdot)\big)_{|\partial \omega}u^b(\varepsilon \cdot)_{|\partial \omega}\, .
\]

\begin{theorem}\label{umk}
With the notation introduced in Proposition \ref{uk}, let $\{\tilde{a}_{n}\}_{n\in\mathbb{N}}$ be the sequence of functions from $\partial \omega$ to $\mathbb{R}$ defined by 
\[
\tilde{a}_{n}\equiv\sum_{k=0}^n \tilde{g}_{n-k}\tilde{v}_{k}\qquad\forall n\in\mathbb{N}\,.
\] Let $\{\tilde{\lambda}_{(n,l)}\}_{(n,l)\in\mathbb{N}^2\,,\;l\le n+1}$ be the family of functions from $\partial \omega$ to $\mathbb{R}$ defined by 
\[
\tilde{\lambda}_{(n,0)}\equiv \tilde{u}_{n}\,,\quad\tilde{\lambda}_{(n,1)}\equiv \tilde{a}_{n}\,,
\] for all $n\in\mathbb{N}$, and
\[
\tilde{\lambda}_{(n,l)}\equiv (-1)^{l-1} \sum_{k=l-1}^n \tilde{a}_{n-k} \sum_{\beta\in(\mathbb{N}\setminus\{0\})^{l-1}\,,\;|\beta|=k}\ \prod_{h=1}^{l-1} r_{\beta_h}
\] for all $n,l\in\mathbb{N}$ with $2\le l\le n+1$. 
  Then there exists $\tilde{\varepsilon}'\in]0,\varepsilon_0]\cap]0,1[$ such that 
  \begin{equation}\label{fuepsmseries}
\nu_{\omega}(t) \cdot \nabla \big(u^a_{\varepsilon}(\varepsilon t)\big) u^b(\varepsilon t)=\sum_{n=0}^\infty\varepsilon^n\sum_{l=0}^{n+1} \frac{\tilde{\lambda}_{(n,l)}(t)}{(r_0+(2\pi)^{-1}\log|\varepsilon|)^{l}} \qquad \forall t \in \partial \omega\, ,
\end{equation} for all $\varepsilon\in]-\tilde{\varepsilon}',\tilde{\varepsilon}'[\setminus\{0\}$. Moreover, the series 
\[
\sum_{n=0}^\infty\varepsilon^n\sum_{l=0}^{n+1} \frac{\tilde{\lambda}_{(n,l)}\eta^l}{(r_0\eta+(2\pi)^{-1})^{l}}\] converges normally in $C^{1,\alpha}(\partial \omega)$  for all $(\varepsilon,\eta)\in]-\tilde{\varepsilon}',\tilde{\varepsilon}'[\times]1/\log\tilde{\varepsilon}',-1/\log\tilde{\varepsilon}'[$.
 \end{theorem}

\medskip

\begin{rem}\label{rem:uepsm}
With the notation of Theorem \ref{umk}, a straightforward computation shows that
\[
\begin{split}
 \tilde{\lambda}_{(0,0)}=&\tilde{u}_{0}=0\, ,  \\
  \tilde{\lambda}_{(0,1)}=&\tilde{a}_{0}=u^a(0) u^b(0) \frac{\partial}{\partial \nu_{\omega}}v^-[\partial \omega, \rho^i_0]\, .
\end{split}
\]
\end{rem}

\subsection{Series expansion of  $\mathrm{Cap}_\Omega(\varepsilon\overline{\omega},u^a,u^b)$}\label{sec5}

We now wish to compute a series expansion of the $(u^a,u^b)$-capacity $\mathrm{Cap}_\Omega(\varepsilon\overline{\omega},u^a,u^b)$. Since $\mathrm{Cap}_\Omega(\varepsilon\overline{\omega},u^a,u^b)$ is represented as a combination of $\int_{\Omega_\varepsilon}\nabla u^a_\varepsilon \cdot \nabla u^b_\varepsilon \, dx$ and of $\int_{\varepsilon \omega}\nabla u^a \cdot \nabla u^b \, dx$, as a first step, we provide an expansion for $\int_{\varepsilon \omega}\nabla u^a \cdot \nabla u^b \, dx$ around $\varepsilon=0$. The term $\int_{\varepsilon \omega}\nabla u^a \cdot \nabla u^b \, dx$ depends analytically on $\varepsilon$. As a consequence, it can be expanded in a power series and we compute such a power series in the following lemma.

\begin{lem}\label{lem:nrguom}
Let $\{\xi_{n}\}_{n\in\mathbb{N}}$ be the sequence of real numbers defined by 
\[
\begin{split}
& \xi_0\equiv0\, , \qquad \xi_1\equiv0\, , \qquad\xi_{n}\equiv\sum_{j=1}^2\sum_{l=0}^{n-2} \int_{\omega} \partial_j u^a_{\#,l+1}(t)\partial_j u^b_{\#,n-l-1}(t)\, dt\qquad \forall n \geq 2\ .
\end{split}
\] 
Then there exists $\varepsilon_\xi\in]0,\varepsilon_0]$ such that 
\[%\label{eq:nrguom}
\int_{\varepsilon \omega}\nabla u^a\cdot \nabla u^b \, dx=\sum_{n=2}^\infty\xi_n \varepsilon^n
\]
for all $\varepsilon\in]-\varepsilon_\xi,\varepsilon_\xi[\setminus\{0\}$. Moreover, 
\[
\xi_2 = \nabla u^a(0)\cdot \nabla u^b(0) m_2(\omega) \, ,
\]
and the series 
\[
\sum_{n=2}^\infty\xi_n \varepsilon^n
\] 
converges absolutely for all $\varepsilon \in]-\varepsilon_\xi,\varepsilon_\xi[$. (The symbol $m_2(\dots)$ denotes the two-dimensional Lebesgue measure of a set).
\end{lem}
{\proof 
By the Theorem of change of variable in integrals, we have
\[
\int_{\varepsilon \omega}\nabla u^a\cdot \nabla u^b  \, dx=\varepsilon^2\int_{ \omega}(\nabla u^a)(\varepsilon t)\cdot (\nabla u^b)(\varepsilon t) \, dt\, ,
\]
for all $\varepsilon \in ]-\varepsilon_0,\varepsilon_0[\setminus\{0\}$. Let $j \in \{1,2\}$. Assumption \eqref{eq:assf}  on the analyticity of $u^a$ and $u^b$ and analyticity results for the composition operator (see  B\"{o}hme and Tomi~\cite[p.~10]{BoTo73}, 
Henry~\cite[p.~29]{He82}, Valent~\cite[Thm.~5.2, p.~44]{Va88}), imply that there exists $\varepsilon_\xi \in ]0,\varepsilon_0]$ such that the maps from $]-\varepsilon_\xi,\varepsilon_\xi[$ to $C^{0,\alpha}(\overline{\omega})$ which take $\varepsilon$ to $(\partial_j u^l)(\varepsilon \cdot)_{|\overline{\omega}}$, $l=a,b$ are real analytic. Also, possibly shrinking $\varepsilon_\xi$, one can verify that  for $\varepsilon \in ]-\varepsilon_\xi,\varepsilon_\xi[\setminus \{0\}$, 
\[
(\partial_j u^l )(\varepsilon t) =\sum_{h=0}^\infty \partial_j u^l_{\#,h+1}(t) \varepsilon^{h} \qquad \forall t \in \overline{\omega}\, , \quad l=a,b\, ,
\]
where the series $\sum_{h=0}^\infty \partial_j u^l_{\#, h+1|\overline{\omega}} \varepsilon^h$ converges normally in $C^{0,\alpha}(\overline{\omega})$  for all $\varepsilon \in ]-\varepsilon_\xi,\varepsilon_\xi[$, for $l=a,b$. Accordingly,
\[
(\partial_j u^a )(\varepsilon t) (\partial_j u^b )(\varepsilon t)=  \sum_{n=0}^\infty \Bigg(\sum_{l=0}^n \partial_j u^a_{\#,l+1}(t)\partial_j u^b_{\#,n-l+1}(t)\Bigg) \varepsilon^n \qquad \forall t \in \overline{\omega}\, , \forall \varepsilon \in ]-\varepsilon_\xi,\varepsilon_\xi[\setminus \{0\}\, .
\]
Moreover, by the continuity of the linear operator from $C^{0,\alpha}(\overline{\omega})$ to $\mathbb{R}$ which takes a function $h$ to its integral $\int_{\omega}h\, dt$, by summing on $j \in \{1,2\}$, possibly taking a smaller $\varepsilon_\xi$, we have that
\begin{equation}\label{eq:nrguom1}
\int_{ \omega}(\nabla u^a)(\varepsilon t)\cdot (\nabla u^b)(\varepsilon t) \, dt=\sum_{n=0}^\infty \bigg(\sum_{j=1}^2\sum_{l=0}^n \int_{\omega}\partial_j u^a_{\#,l+1}(t)\partial_j u^b_{\#,n-l+1}(t)\, dt\bigg) \varepsilon^n\, ,
\end{equation}
for all $\varepsilon \in ]-\varepsilon_\xi,\varepsilon_\xi[\setminus \{0\}$. Also,
\[
\sum_{j=1}^2 \int_{\omega}\partial_j u^a_{\#,1}(t)\partial_j u^b_{\#,1}(t)\, dt=  \nabla u^a(0)\cdot \nabla u^b(0) m_2(\omega)\, .
\]
Finally, by multiplying equation \eqref{eq:nrguom1} by $\varepsilon^2$, we deduce the validity of the lemma.\qed}

\medskip

We are now ready to deduce the validity of our main result of this section on the asymptotic behavior of $\mathrm{Cap}_\Omega(\varepsilon\overline{\omega},u^a, u^b)$.

\begin{theorem}\label{capk}
With the notation introduced in Proposition \ref{uk}, Theorem \ref{umk} and Lemma \ref{lem:nrguom}, let $\{c_{(n,l)}\}_{\substack{(n,l)\in\mathbb{N}^2\\\;l\le n+1}}$ be the family of real numbers defined by
\[
c_{(n,l)}\equiv-\int_{\partial \omega}\tilde{\lambda}_{(n,l)}\, d\sigma+\delta_{0,l}\xi_n\, ,
\] 
for all $n,l\in\mathbb{N}$ with $l\le n+1$ (where $\delta_{0,l}=1$ if $l=0$ and $\delta_{0,l}=0$ if $l\neq0$).   Then there exists $\varepsilon_\mathrm{c}\in]0,\varepsilon_0]\cap]0,1[$ such that  the series 
\[
\sum_{n=0}^\infty\varepsilon^n\sum_{l=0}^{n+1} \frac{c_{(n,l)}\eta^l}{(r_0\eta+(2\pi)^{-1})^{l}}\] converges  absolutely for all $(\varepsilon,\eta)\in]-\varepsilon_\mathrm{c},\varepsilon_\mathrm{c}[\times]1/\log\varepsilon_\mathrm{c},-1/\log\varepsilon_\mathrm{c}[$ and that

\[%  \begin{equation}\label{cepsmseries}
\mathrm{Cap}_\Omega(\varepsilon\overline{\omega},u^a,u^b)=\sum_{n=0}^\infty\varepsilon^n\sum_{l=0}^{n+1} \frac{c_{(n,l)}}{(r_0+(2\pi)^{-1}\log|\varepsilon|)^{l}}
\]%\end{equation} 
for all $\varepsilon\in]-\varepsilon_{\mathrm{c}},\varepsilon_\mathrm{c}[\setminus\{0\}$.
 \end{theorem}
  \proof  By integrating over $\partial \omega$  formula \eqref{fuepsmseries} and adding the coefficients of Lemma \ref{lem:nrguom} and by Theorem \ref{umk}, we  immediately deduce the validity of the statement.\qed
 
 \medskip
 
\begin{rem}\label{rem:cepsm}
With the notation of Theorem \ref{capk}, we observe that Remark \ref{rem:uepsm} and a straightforward computation   based on Folland \cite[Lem.~3.30]{Fo95} imply that
\[
 c_{(0,0)}=0\, ,  \qquad
c_{(0,1)}= -u^a(0)u^b(0)\, .
\]
Moreover, if we denote by  $H^o_0$   the unique solution in $C^{1,\alpha}(\overline{\Omega})$ of
\[% \begin{equation}\label{auxbvpHo}
\left\{
\begin{array}{ll}
\Delta H^o_0=0&\text{in }\Omega\,,\\
H^o_{0}(x)=S(x)&\text{for all }x\in\partial\Omega\,,
\end{array}
\right.
\]%\end{equation} 
and by $H^i_0$  the unique solution  in $C^{1,\alpha}_{\mathrm{loc}}(\mathbb{R}^2\setminus\omega)$ of
\[%\begin{equation}\label{auxbvpHi}
\left\{
\begin{array}{ll}
\Delta H^i_0=0&\text{in }\mathbb{R}^2\setminus\overline{\omega}\,,\\
H^i_{0}(t)=S(t)&\text{for all }t\in\partial\omega\,,\\
\sup_{t\in\mathbb{R}^2\setminus\omega}|H^i_0(t)|<+\infty\,,
\end{array} 
\right.
\]%\end{equation}
then by \cite[Prop.~7.3]{DaMuRo15} we have
\[
r_0= \lim_{t\to\infty}H^i_0(t)-H^o_0(0)\,.
\]
Accordingly,
\begin{equation}\label{eq:capepsfirst}
\begin{split}
\mathrm{Cap}_\Omega(\varepsilon\overline{\omega},u^a,u^b)=& -\frac{u^a(0)u^b(0)}{\lim_{t\to\infty}H^i_0(t)-H^o_0(0)+(2\pi)^{-1}\log|\varepsilon|}\\
&+\varepsilon\bigg(\sum_{n=1}^\infty\varepsilon^{n-1}\sum_{l=0}^{n+1} \frac{c_{(n,l)}}{(\lim_{t\to\infty}H^i_0(t)-H^o_0(0)+(2\pi)^{-1}\log|\varepsilon|)^{l}}\bigg)
\end{split}
\end{equation}
for all $\varepsilon\in]-\varepsilon_{\mathrm{c}},\varepsilon_\mathrm{c}[\setminus\{0\}$. 
\end{rem}

\subsection{Asymptotic behavior of $\mathrm{Cap}_\Omega(\varepsilon\overline{\omega},u^a,u^b)$ under vanishing assumption for $u^a$ and $u^b$}

The aim of this subsection is to discuss the behavior of $\mathrm{Cap}_\Omega(\varepsilon\overline{\omega},u^a,u^b)$ under the assumption that $u^a$ and $u^b$ vanish together with their derivatives up to a certain order. We work as in \cite[\S 5.1]{AbBoLeMu} and we assume that there exist $\overline{k}^a, \overline{k}^b \in \mathbb{N}\setminus \{0\}$ such that
\begin{equation}\label{eq:vanu}
D^\gamma u^l(0)=0 \quad \forall |\gamma| <\overline{k}^l\, , \qquad D^{\beta^l} u^l(0)\neq 0 \quad \mbox{for some $\beta^l \in \mathbb{N}^2$ with $|\beta^l|=\overline{k}^l$}\, ,\qquad l=a,b\,.
\end{equation}
By condition \eqref{eq:vanu} and Proposition \ref{thetak} we have that
\begin{equation}\label{eq:vantheta}
(\theta^o_k,\theta^i_k)=(0,0) \quad \forall k < \overline{k}^a\, , \qquad \theta^o_{\overline{k}^a}=0\, ,
\end{equation}
and that $\theta^i_{\overline{k}^a}$ is the unique solution in $C^{1,\alpha}(\partial\omega)_0$ of  
\begin{equation}\label{eq:thetaooverk}
\begin{split}
\frac{1}{2}\theta^i_{\overline{k}^a}(t)-W_{\omega}[\theta^i_{\overline{k}^a}](t)&=\overline{k}^a! \Bigg (u^a_{\#,\overline{k}^a}(t)-\int_{\partial\omega}u^a_{\#,\overline{k}^a}\rho^i_{0}\,d\sigma\Bigg) \qquad \forall t\in\partial\omega\,.
\end{split}
\end{equation}
By \eqref{eq:vantheta}, \eqref{eq:thetaooverk} and  by Proposition \ref{uk}, we deduce that
\begin{equation}\label{eq:vanum}
u^a_{\mathrm{m},k}=0 \qquad \forall k <\overline{k}^a\, ,\qquad u^a_{\mathrm{m},\overline{k}^a}=-\frac{1}{\overline{k}^a!}w^-[\partial \omega, \theta^i_{\overline{k}^a}]\, .
\end{equation}
Classical potential theory implies that $u^a_{\mathrm{m},\overline{k}^a}$ is the unique solution in $C^{1,\alpha}_{\mathrm{loc}}(\mathbb{R}^2 \setminus \omega)$ of 
\[%\begin{equation}\label{eq:bvpumk}
\left\{
\begin{array}{ll}
\Delta u^a_{\mathrm{m},\overline{k}^a}=0&\text{in }\mathbb{R}^2\setminus\overline{\omega}\,,\\
u^a_{\mathrm{m},\overline{k}^a}(t)=u^a_{\#,\overline{k}^a}(t)-\int_{\partial\omega}u^a_{\#,\overline{k}^a}\rho^i_{0}\,d\sigma&\text{for all }t\in\partial\omega\,,\\
\sup_{t\in\mathbb{R}^2\setminus\omega}|u^a_{\mathrm{m},\overline{k}^a}(t)|<+\infty\,.
\end{array} 
\right.
\]%\end{equation}
Also, by assumption \eqref{eq:vanu} and Proposition \ref{uk} we deduce that
\begin{equation}\label{eq:vang}
g^a_{k}=0 \quad \forall k <\overline{k}^a\, ,\qquad g_{\overline{k}^a}=\frac{1}{\overline{k}^a!}\sum_{h=0}^{\overline{k}^a}  \binom{{\overline{k}^a}}{h}\int_{\partial\omega}s_1^hs_2^{{\overline{k}^a}-h}(\partial_1^h \partial_2^{{\overline{k}^a}-h}u^a)(0)\rho^i_{0}(s)\,d\sigma_s=\int_{\partial\omega}u^a_{\#,\overline{k}^a}\rho^i_{0}\,d\sigma\, .
\end{equation}
Then by \eqref{eq:vanu} and by Propostion \ref{funk} we verify that
\begin{equation}\label{eq:vanusharp}
u^l_{\#,k}=0 \qquad \forall k <\overline{k}^l\, , \qquad l=a,b\, .
\end{equation}
As a consequence, Proposition \ref{funk} and equations \eqref{eq:vanum}, \eqref{eq:vanusharp} imply
\begin{equation}\label{eq:vantildeu}
\tilde{u}_{k}=0 \qquad \forall k <\overline{k}^a+\overline{k}^b\, ,\qquad \tilde{u}_{\overline{k}^a+\overline{k}^b}=\Bigg( \frac{\partial u^a_{\mathrm{m},\overline{k}^a}}{\partial \nu_\omega}\Bigg)u_{\#,\overline{k}^b|\partial \omega}\, .
\end{equation}
Furthermore, by \eqref{eq:vang} and \eqref{eq:vanusharp} we have
\begin{equation}\label{eq:vantildeg}
\tilde{g}_{k}=0 \qquad \forall k <\overline{k}^a+\overline{k}^b\, ,\qquad \tilde{g}_{\overline{k}^a+\overline{k}^b}=g^a_{\overline{k}^a}u^b_{\#,\overline{k}^b|\partial \omega}=
 \bigg(\int_{\partial\omega}u^a_{\#,\overline{k}^a}\rho^i_{0}\,d\sigma \bigg)  u^b_{\#,\overline{k}^b|\partial \omega} \, .
\end{equation}
To compute the coefficients of the expansion of the $(u^a,u^b)$-capacity $\mathrm{Cap}_\Omega(\varepsilon\overline{\omega},u^a,u^b)$,  as an intermediate step, we consider the quantities $\tilde{a}_{n}, \tilde{\lambda}_{(n,l)}$ introduced in Theorem \ref{umk} for representing the behavior of $\nu_{\omega}(\cdot) \cdot \nabla \big(u^a_{\varepsilon}(\varepsilon\cdot)\big)_{|\partial \omega}u^b(\varepsilon \cdot)$. A straightforward computation based on \eqref{eq:vantildeu}, \eqref{eq:vantildeg} implies that
\[
\tilde{a}_{n}=0 \qquad \forall n <\overline{k}^a+\overline{k}^b\, ,\qquad \tilde{a}_{\overline{k}^a+\overline{k}^b}=\tilde{g}_{\overline{k}^a+\overline{k}^b}\tilde{v}_0=  \tilde{v}_0
 \bigg(\int_{\partial\omega}u^a_{\#,\overline{k}^a}\rho^i_{0}\,d\sigma \bigg)  u^b_{\#,\overline{k}^b|\partial \omega}\, ,
 \]
and accordingly
\begin{equation}\label{eq:vantilde0}
\tilde{\lambda}_{(n,0)}=0 \qquad \forall n <\overline{k}^a+\overline{k}^b\, ,\qquad \tilde{\lambda}_{\overline{k}^a+\overline{k}^b,0}=\tilde{u}_{\overline{k}^a+\overline{k}^b}= \Bigg( \frac{\partial u^a_{\mathrm{m},\overline{k}^a}}{\partial \nu_\omega}\Bigg) u^b_{\#,\overline{k}^b|\partial \omega}\, ,
\end{equation}
\begin{equation}\label{eq:vantilde1}
\tilde{\lambda}_{(n,1)}=0 \qquad \forall n <\overline{k}^a+\overline{k}^b\, ,\qquad \tilde{\lambda}_{\overline{k}^a+\overline{k}^b,1}=\tilde{a}_{\overline{k}^a+\overline{k}^b}=  \tilde{v}_0
 \bigg( \int_{\partial\omega}u^a_{\#,\overline{k}^a}\rho^i_{0}\,d\sigma\bigg) u^b_{\#,\overline{k}^b|\partial \omega}\, ,
\end{equation}
and
\begin{equation}\label{eq:vantildel}
\tilde{\lambda}_{(n,l)}=0 \qquad \text{$\forall (n,l)$ such that  $n-l+1<\overline{k}^a+\overline{k}^b$ and that $2\leq l \leq n+1$}\, .
\end{equation}
In particular, $\tilde{\lambda}_{(n,l)}=0$ for all $(n,l)$ such that  $n<\overline{k}^a+\overline{k}^b+1$ and that $2\leq l \leq n+1$. Furthermore, by a simple computation we have that
\[
\xi_{n}=0 \qquad \forall n < \overline{k}^a+\overline{k}^b\, , \qquad \xi_{ \overline{k}^a+\overline{k}^b}= \int_{\omega} \nabla u^a_{\#,\overline{k}^a}\cdot \nabla u^b_{\#,\overline{k}^b} \, dt\, .
\]
Finally, by Theorem \ref{capk} and by integrating equalities \eqref{eq:vantilde0}-\eqref{eq:vantildel}, we obtain
\[
\begin{split}
&c_{(n,0)}=0 \qquad \forall n <\overline{k}^a+\overline{k}^b\, ,\\ 
&c_{\overline{k}^a+\overline{k}^b,0}=-\int_{\partial \omega}\tilde{u}_{\overline{k}^a+\overline{k}^b}\, d\sigma+\int_{\omega}  \nabla u^a_{\#,\overline{k}^a}\cdot \nabla u^b_{\#,\overline{k}^b}\, dt\\
&\qquad=-\int_{\partial \omega}  \Bigg( \frac{\partial u^a_{\mathrm{m},\overline{k}^a}}{\partial \nu_\omega}\Bigg) u^b_{\#,\overline{k}^b|\partial \omega}\, d\sigma+\int_{\omega}  \nabla u^a_{\#,\overline{k}^a}\cdot \nabla u^b_{\#,\overline{k}^b}\, dt\,, 
\end{split}
\]
\[
c_{(n,1)}=0 \qquad \forall n <\overline{k}^a+\overline{k}^b\, ,\qquad c_{\overline{k}^a+\overline{k}^b,1}=-\int_{\partial \omega}\tilde{a}_{\overline{k}^a+\overline{k}^b}\,d\sigma=- \int_{\partial\omega}u^a_{\#,\overline{k}^a}\rho^i_{0}\,d\sigma \int_{\partial \omega} \tilde{v}_0
 u^b_{\#,\overline{k}^b|\partial \omega}\, d\sigma\, , 
\]
and
\[
c_{(n,l)}=0 \qquad \text{$\forall (n,l)$ such that  $n-l+1<\overline{k}^a+\overline{k}^b$ and that $2\leq l \leq n+1$}\, .
\]
In particular, $c_{(n,l)}=0$ for all $(n,l)$ such that  $n<\overline{k}^a+\overline{k}^b+1$ and that $2\leq l \leq n+1$.  
We also note that, since $u^a_{\mathrm{m},\overline{k}^a}=-\frac{1}{\overline{k}^a!}w^-[\partial \omega, \theta^i_{\overline{k}^a}]$, then $u^a_{\mathrm{m},\overline{k}^a}$ is harmonic at infinity (see  \eqref{eq:vanum}).
As a consequence, by the decay properties of its radial derivative (see  Folland \cite[Prop.~2.75]{Fo95}) and by the Divergence Theorem one verifies that
\[
\int_{\partial \omega}\frac{\partial u^a_{\mathrm{m},\overline{k}^a}}{\partial \nu_\omega}\, d\sigma=0\, .
\]
Then, for $l=a,b$, we denote by $\mathsf{u}^l_{\overline{k}^l}$ the unique solution in $C^{1,\alpha}_{\mathrm{loc}}(\mathbb{R}^2\setminus \omega)$ of
\begin{equation}\label{eq:ulk}
\left\{
\begin{array}{ll}
\Delta \mathsf{u}^l_{\overline{k}^l}=0&\text{in }\mathbb{R}^2\setminus\overline{\omega}\,,\\
\mathsf{u}^l_{\overline{k}^l}(t)=u^{l}_{\#,\overline{k}^l}(t)&\text{for all }t\in\partial\omega\,,\\
\sup_{t\in\mathbb{R}^2\setminus\omega}|\mathsf{u}^l_{\overline{k}^l}(t)|<+\infty\,.
\end{array} 
\right.
\end{equation}
Then clearly
\[
\mathsf{u}^a_{\overline{k}^a}=u^a_{\mathrm{m},\overline{k}^a}+\int_{\partial\omega}u^a_{\#,\overline{k}^a}\rho^i_{0}\,d\sigma\, ,
\]
and thus
\[
-\int_{\partial \omega}  \Bigg( \frac{\partial u^a_{\mathrm{m},\overline{k}^a}}{\partial \nu_\omega}\Bigg) u^b_{\#,\overline{k}^b|\partial \omega} \, d\sigma=-\int_{\partial \omega}  \Bigg( \frac{\partial \mathsf{u}^a_{\overline{k}^a}}{\partial \nu_\omega}\Bigg) \mathsf{u}^b_{\overline{k}^b|\partial \omega} \, d\sigma\, .
\]

On the other hand,  the harmonicity at infinity of $\mathsf{u}^a_{\overline{k}^a}, \mathsf{u}^b_{\overline{k}^b}$ and the Divergence Theorem imply that
\[
\int_{\mathbb{R}^2 \setminus \overline{\omega}}\nabla \mathsf{u}^a_{\overline{k}^a}\cdot \nabla \mathsf{u}^b_{\overline{k}^b}\, dt=-\int_{\partial \omega}  \Bigg( \frac{\partial \mathsf{u}^a_{\overline{k}^a}}{\partial \nu_\omega}\Bigg) \mathsf{u}^b_{\overline{k}^b|\partial \omega} \, d\sigma
\]
(see  Folland \cite[p.~118]{Fo95}). Hence,
\[%\begin{equation}\label{eq:ibp}
-\int_{\partial \omega}  \Bigg( \frac{\partial u^a_{\mathrm{m},\overline{k}^a}}{\partial \nu_\omega}\Bigg) u^b_{\#,\overline{k}^b|\partial \omega} \, d\sigma=\int_{\mathbb{R}^2 \setminus \overline{\omega}}\nabla \mathsf{u}^a_{\overline{k}^a}\cdot \nabla \mathsf{u}^b_{\overline{k}^b}\, dt\, .
\]%\end{equation}
We now consider  
\[
- \int_{\partial\omega}u^a_{\#,\overline{k}^a}\rho^i_{0}\,d\sigma \int_{\partial \omega} \tilde{v}_0
 u^b_{\#,\overline{k}^b|\partial \omega}\, d\sigma\, .
 \]
First, we  note that
\[
\tilde{v}_0= \nu_{\omega}\cdot\nabla v_{\mathrm{m},0|\partial \omega}=\nu_{\omega}\cdot\nabla v^-[\partial \omega,\rho^i_0]_{|\partial \omega}\, .
\]
On the other hand, Proposition \ref{rhok} and the jump formula for the normal derivative of the single layer potential imply that
\[
\nu_{\omega}\cdot\nabla v^-[\partial \omega,\rho^i_0]_{|\partial \omega}=\frac{1}{2}\rho^i_0+W^*_{\omega}[ \rho^i_0]=\frac{1}{2}\rho^i_0+\frac{1}{2}\rho^i_0=\rho^i_0\, .
\]
Accordingly,
\[
\int_{\partial \omega} \tilde{v}_0
 u^b_{\#,\overline{k}^b|\partial \omega}\, d\sigma=\int_{\partial \omega}u^b_{\#,\overline{k}^b}\rho^i_{0}\,d\sigma\, .
\]
By \cite[Proof of Lem.~7.2]{DaMuRo15}, we have
\[
\int_{\partial \omega}u^l_{\#,\overline{k}^l}\rho^i_{0}\,d\sigma=\lim_{t\to \infty}\mathsf{u}^l_{\overline{k}^l}(t)\qquad l=a,b\, , 
\]
which implies
\[
- \int_{\partial\omega}u^a_{\#,\overline{k}^a}\rho^i_{0}\,d\sigma \int_{\partial \omega} \tilde{v}_0
 u^b_{\#,\overline{k}^b|\partial \omega}\, d\sigma=-\bigg(\lim_{t\to \infty}\mathsf{u}^a_{\overline{k}^a}(t)\bigg)\bigg(\lim_{t\to \infty}\mathsf{u}^b_{\overline{k}^b}(t)\bigg)\, .
\]
Under assumption \eqref{eq:vanu}, by Remark \ref{rem:cepsm} and formula \eqref{eq:capepsfirst}, we can now deduce the validity of the following.
 
 \begin{theorem}\label{thm:cepsmseries}
Let assumption \eqref{eq:vanu} hold. Then
 \begin{equation}\label{cepsmseries}
\begin{split}
\mathrm{Cap}_\Omega(\varepsilon\overline{\omega},u^a,u^b)=&\varepsilon^{\overline{k}^a+\overline{k}^b}\Bigg(\int_{\mathbb{R}^2 \setminus \overline{\omega}}\nabla \mathsf{u}^a_{\overline{k}^a}\cdot \nabla \mathsf{u}^b_{\overline{k}^b}\, dt+\int_{\omega}  \nabla u^a_{\#,\overline{k}^a}\cdot \nabla u^b_{\#,\overline{k}^b}\, dt\\
&\qquad \qquad -\frac{\Big(\lim_{t\to \infty}\mathsf{u}^a_{\overline{k}^a}(t)\Big)\Big(\lim_{t\to \infty}\mathsf{u}^b_{\overline{k}^b}(t)\Big)}{(\lim_{t\to\infty}H^i_0(t)-H^o_0(0)+(2\pi)^{-1}\log|\varepsilon|)}\Bigg) \\
 &+\sum_{n=\overline{k}^a+\overline{k}^b+1}^\infty\varepsilon^n\sum_{l=0}^{n-(\overline{k}^a+\overline{k}^b)+1} \frac{c_{(n,l)}}{(\lim_{t\to\infty}H^i_0(t)-H^o_0(0)+(2\pi)^{-1}\log|\varepsilon|)^{l}}\, ,
\end{split}
\end{equation} 
for all $\varepsilon\in]-\varepsilon_{\mathrm{c}},\varepsilon_\mathrm{c}[\setminus\{0\}$. 
\end{theorem}

\begin{rem}\label{rem:cepsmseries}
Therefore, by \eqref{cepsmseries} we have
 \[%\begin{equation}\label{cepsmseriesbis}
\mathrm{Cap}_\Omega(\varepsilon\overline{\omega},u^a,u^b)=\varepsilon^{\overline{k}^a+\overline{k}^b}\Bigg(\int_{\mathbb{R}^2 \setminus \overline{\omega}}\nabla \mathsf{u}^a_{\overline{k}^a}\cdot \nabla \mathsf{u}^b_{\overline{k}^b}\, dt+\int_{\omega}  \nabla u^a_{\#,\overline{k}^a}\cdot \nabla u^b_{\#,\overline{k}^b}\, dt\Bigg)+o(\varepsilon^{\overline{k}^a+\overline{k}^b}) \qquad \text{as }\varepsilon \to 0\, .
\]%\end{equation}
Moreover, we note that the terms $\int_{\mathbb{R}^2 \setminus \overline{\omega}}\nabla \mathsf{u}^a_{\overline{k}^a}\cdot \nabla \mathsf{u}^b_{\overline{k}^b}\, dt$ and $\int_{\omega}  \nabla u^a_{\#,\overline{k}^a}\cdot \nabla u^b_{\#,\overline{k}^b}\, dt$ depend  both on the geometrical properties of the set $\omega$ and on the behavior at $0$ of the functions $u^a$ and $u^b$ (but not on $\Omega$).
\end{rem}

\section{Perturbation of eigenvalues}

In order to find an approximation of the perturbed eigenvalues $\{\lambda_j^\eps\}$, we use a slight modification of a lemma from G. Courtois \cite{Courtois1995}, itself based on the work of Y. Colin de Verdière \cite{ColindeV1986}.  For completeness,  we give its proof in  Appendix \ref{app:smallEVs}.

\begin{prop}\label{p:appEV}
Let $(\mathcal H, \|\cdot\|)$  be a Hilbert space and $q$ be a quadratic form, semi-bounded from below (not necessarily positive), with
domain $\mathcal D$ dense in $\mathcal H$ and with discrete spectrum $\{ \nu_i \}_{i\geq1}$. Let $\{ g_i \}_{i\geq1}$ be an orthonormal basis of eigenvectors of $q$. Let $N$ and $m$ be positive integers, $F$ an $m$-dimensional subspace of $\mathcal D$ and $\{ \xi_i^F\}_{i=1}^m$ the eigenvalues of the restriction of $q$ to $F$.

Assume that there exist positive constants $\gamma$ and $\delta$ such that
\begin{itemize}
 \item[(H1)] $ 0<\delta<\gamma/\sqrt2$;
 \item[(H2)] for all $i\in\{1,\dots,m\}$, $|\nu_{N+i-1}|\le\gamma$, $\nu_{N+m}\ge \gamma$ and, if $N\ge2$, $\nu_{N-1}\le-\gamma$;
 \item[(H3)] $|q(\varphi,g)|\leq \delta\, \|\varphi \|\,\|g\|$ for all $g\in\mathcal D$ and $\varphi \in F$. 
\end{itemize}
Then we have
\begin{itemize}
 \item[(i)] $\left|\nu_{N+i-1}- \xi_i^F \right|\le\frac{ 4}{\gamma}\delta^2$ for all $i=1,\ldots,m$; 
 \item[(ii)] $\left\| \Pi_N - \mathbb{I}\right\|_{\mathcal L(F,\mathcal H)} \leq { \sqrt 2}\delta/\gamma$,  where $\Pi_N$ is the projection onto the subspace of $\mathcal D$ spanned by $\{g_N,\ldots,g_{N+m-1}\}$. 
\end{itemize}
\end{prop}

Our subsequent analysis is close to  \cite[Proof of Theorem 1.2]{Courtois1995},  except that we replace the standard capacity with our generalized $u$-capacity (or $(u,v)$-capacity). Before proceeding, we recall the following crucial result from \cite{AbFeHiLe2019}. 
\begin{lem}(\cite[Lemma A.1]{AbFeHiLe2019})\label{l:a1}
 For any $f\in H^1_0(\Omega)$ we have
 \[
  \int_{\Omega} |V_{K_\eps,f}|^2 \, dx= o(\mathrm{Cap}_\Omega(K_\eps,f)) \quad \text{as }\eps\to0.
 \]
\end{lem}

We also establish a series of preparatory lemmas. For $\eps>0$, we denote by $\Pi_\eps$ the linear mapping 
\begin{equation*}
\begin{array}{cccc}
	\Pi_\eps:& H^1_0(\Omega)& \to     & H^1_0(\Omega\setminus K_\eps)\\
			 & u            & \mapsto & u-V_{K_\eps,u}.
\end{array}
\end{equation*}
Note that $\Pi_\eps$ is the $q_\Omega$-orthogonal projection described in Remark \ref{rem:geom}.
We recall the quantity $\chi_\eps,$ defined in the statement of Theorem \ref{thm:approxEVs}:
\begin{equation}
\label{eq:chiEps}
	\chi_\eps^2\equiv\sup\{\mathrm{Cap}_\Omega(K_\eps,u)\,:\,u\in E(\lambda_N)\mbox{ and } \|u\|=1\}
\end{equation}
To simplify notation, we denote by $V_{u}^\eps$ the potential $V_{K_\eps,u}$ for all $u\in H_0^1(\Omega)$.

\begin{lem} \label{l:error} As $\eps\to0$, $\chi_\eps\to0$.
\end{lem} 

\begin{proof} Let us pick $u\in E(\lambda_N)$ such that $\|u\|=1$. We write $u=\sum_{i=1}^m c_i u_{N+i-1}$. Then
\begin{align*}
	\mathrm{Cap}_\Omega(K_\eps,u)&=\left|\sum_{1\le i,j\le m}c_ic_j\int_\Omega \nabla V_{u_{N+i-1}}^{\eps}\cdot \nabla V_{u_{N+j-1}}^{\eps}\, dx\right|\\
	&\le \sum_{1\le i,j\le m}|c_i||c_j|\left(\int_\Omega |\nabla V_{u_{N+i-1}}^{ \eps }|^2\, dx\right)^{\frac12}\left(\int_\Omega |\nabla V_{u_{N+j-1}}^{ \eps }|^2\, dx\right)^{\frac12}\\
	&=\left(\sum_{i=1}^m |c_i|\left(\int_\Omega |\nabla V_{u_{N+i-1}}^{ \eps }|^2\, dx\right)^{\frac12}\right)^2\\
	\le &m\,\left(\max_{1\le i\le m}\int_\Omega |\nabla V_{u_{N+i-1}}^{ \eps }|^2\, dx\right)\sum_{i=1}^m c_i^2=m\max_{1\le i\le m}\mathrm{Cap}_\Omega(K_\eps,u_{N+i-1}).
\end{align*}
Since $\mathrm{Cap}_\Omega(K_\eps,u_{N+i-1})\to0$ for all $1\le i\le m$, the result follows.
\end{proof}

\begin{lem} \label{l:norm} As $\eps\to0$, $M_\eps\equiv\left\|\mathbb I-\Pi_\eps\right\|_{{\mathcal L(E(\lambda_N),L^2(\Omega))}}=o(\chi_\eps)$.
\end{lem}

\begin{proof} Let $u\in E(\lambda_N)$ such that $\|u\|=1$. By definition, we have $(\mathbb I-\Pi_\eps)u=V_u^\eps$. We find
\begin{align*}
 	\|V_u^\eps\|&\le \sum_{i=1}^m |c_i|\|V_{N+i-1}^{ \eps }\| \le \left(\sum_{i=1}^m c_i^2\right)^{\frac12}\left(\sum_{i=1}^m \|V_{N+i-1}^{ \eps }\|^2\right)^{\frac12}\\
 	&=\left(\sum_{i=1}^m \mathrm{Cap}_\Omega(K_\eps,u_{N+i-1})\frac{\|V_{N+i-1}^{ \eps }\|^2}{\mathrm{Cap}_\Omega(K_\eps,u_{N+i-1})}\right)^{\frac12}\\ &\le \sqrt{m}\,\chi_\eps\,\max_{1\le i\le m}\frac{\|V_{N+i-1}^{ \eps }\|}{\mathrm{Cap}_\Omega(K_\eps,u_{N+i-1})^{(1/2)}}.
\end{align*}
The last term is $o(\chi_\eps)$ according to Lemma \ref{l:a1}.
\end{proof}

Lemma \ref{l:norm} implies in particular that $M_\eps<1$,  so that the restriction of $\Pi_\eps$ to $E_N$ is injective, for $\eps$ small enough. We will always suppose this to be the case in the rest of this section.

\subsection{Application of the abstract lemma}\label{sec:appLemma}

We are going to apply Proposition \ref{p:appEV} in the following way. For $\eps>0$ small enough, we introduce the following set of definitions \eqref{notfirst}--\eqref{notlast}:
\begin{align}
 &\mathcal H_\eps \equiv L^2(\Omega\setminus K_\eps);\label{notfirst}\\
 &\mathcal D_\eps \equiv H^1_0(\Omega\setminus K_\eps);\\
&q_\eps(u)\equiv \int_{\Omega\setminus K_\eps} |\nabla u|^2\, dx-\lambda_N\int_{\Omega\setminus K_\eps} u^2\, dx\text{ for all }u\in\mathcal D_\eps;\\
&F_{\eps}\equiv \Pi_\eps(E(\lambda_N)). \label{notlast}
\end{align}

By construction, the eigenvalues of $q_\eps$ are $\{\lambda_i^\eps-\lambda_N\}_{i\ge1}$. We use the notation $\nu_i^\eps\equiv \lambda_i^\eps-\lambda_N$. Since $\lambda_i^\eps\to\lambda_i$  for all $i\in\N^*$ and since $\lambda_N$ is of multiplicity $m$, the assumption (H2) is fulfilled for $\eps>0$ small enough if we take, for instance,
\begin{equation*}
	\gamma\equiv{\frac12}\min\{\lambda_N-\lambda_{N-1},\lambda_{N+m}-\lambda_{N+m-1}\}
\end{equation*}
when $N\ge2$ and, when $N=1$ (in which case $m=1$),
\begin{equation*}
	\gamma\equiv{ \frac12}
	\left(\lambda_{2}-\lambda_{1}\right).
\end{equation*}

It remains to check whether condition (H3) is satisfied. Let us choose $v\in F_\eps$ and $w\in \mathcal D_\eps$. There exists a unique $u\in E(\lambda_N)$ such that $v=\Pi_\eps u$. We have
\begin{align*} 
q_\eps(v,w)&=q_\Omega(v,w)-\lambda_N\langle v,w\rangle\\
		   &=q_\Omega(u,w)-\lambda_N\langle u,w\rangle-q_\Omega(V_u^\eps,w)+\lambda_N\langle V_u^\eps,w\rangle\\
		   &=\lambda_N\langle V_u^\eps,w\rangle.
\end{align*}
We have used the facts that $u$ is an eigenfunction of $q_\Omega$ and that $V_u^\eps$ is $q_\Omega$-orthogonal to $\mathcal D_\eps$. We obtain
\begin{equation*}
	|q_\eps(v,w)|\le\lambda_N\|V_u^\eps\|\|w\| \le \lambda_N M_\eps \|u\|\|w\|\le\lambda_N\frac{M_\eps}{1-M_\eps}\|v\|\|w\|.
\end{equation*}
Lemma \ref{l:norm} then implies
\begin{equation*}
	|q_\eps(v,w)|\le o(\chi_\eps)\|v\|\|w\|.
\end{equation*}

We can now apply Proposition \ref{p:appEV}, which tells us that for $1\le i\le m$,
\begin{equation*}
	\lambda_{N+i-1}^\eps=\lambda_N+\xi_i^\eps +o(\chi_\eps^2),
\end{equation*}
where $\{ \xi_i^\eps \}_{i=1}^m$ are the eigenvalues of the restriction of $q_\eps$ to $F_\eps$. 
 
\subsection{Analysis of the restricted quadratic form. Proof of Theorem \ref{thm:approxEVs}}

In  order to give a complete proof of Theorem \ref{thm:approxEVs}, it remains to show that for $1\le i\le m$,
\begin{equation*}
	|\mu_i^\eps- \xi_i^\eps |=o(\chi_\eps^2).
\end{equation*}
Let us recall that $\{\mu_i^\eps\}_{i=1}^m$ are the eigenvalues of the quadratic form $r_\eps$ defined on $E(\lambda_N)$ by
\begin{equation*}				
	r_\eps(u,v) \equiv q_\Omega(V_u^\eps,V_v^\eps)-\lambda_N\langle V_u^\eps,V_v^\eps\rangle.
\end{equation*}
Note that, from Lemma \ref{l:a1}, $\mu_i^\eps=O\left(\chi_\eps^2\right)$ as $\eps\to0$.
\begin{lem} \label{l:restrict}
For all $u,v\in E(\lambda_N)$, 
\begin{equation*}
	q_\eps\left(\Pi_\eps u,\Pi_\eps v\right)=r_\eps(u,v).
\end{equation*}
\end{lem}

\begin{proof} We have
 \begin{equation*}
 	q_\eps\left(\Pi_\eps u,\Pi_\eps v\right)=q_\Omega(u-V_u^\eps,v-V_v^\eps)-\lambda_N\langle u-V_u^\eps,v-V_v^\eps \rangle.
 \end{equation*}
 Since $V_u^\eps$ is $q_\Omega$-orthogonal to $\mathcal D_\eps$,
 \begin{equation*}
 	q_\eps\left(\Pi_\eps u,\Pi_\eps v\right)=q_\Omega(u,v-V_v^\eps)-\lambda_N\langle u,v-V_v^\eps \rangle+\lambda_N\langle V_u^\eps,v-V_v^\eps \rangle
 \end{equation*}
 and since $u$ is an eigenfunction of $q_\Omega$ associated with
%associée à 
 $\lambda_N$,
 \begin{equation*}
 	q_\eps\left(\Pi_\eps u,\Pi_\eps v\right)=\lambda_N\langle V_u^\eps,v \rangle-\lambda_N\langle V_u^\eps,V_v^\eps \rangle.
 \end{equation*} 
 We finally use the fact that $v$ is an eigenfunction associated with $\lambda_N$ and that $V_v^\eps$ is the projection of $v$ on the $q_\Omega$-orthogonal complement of $\mathcal D_\eps$:
  \begin{equation*}
 	q_\eps\left(\Pi_\eps u,\Pi_\eps v\right)=q_\Omega(V_u^\eps,v)-\lambda_N\langle V_u^\eps,V_v^\eps \rangle=q_\Omega(V_u^\eps,V_v^\eps)-\lambda_N\langle V_u^\eps,V_v^\eps \rangle=r_\eps(u,v).\qedhere
 \end{equation*} 
\end{proof}

Let us now introduce the notation $v_i^\eps\equiv \Pi_\eps u_{N+i-1}$. Since $\{u_{N+i-1}\}_{1\le i\le m}$ is an orthonormal basis of $E(\lambda_N)$ and $\Pi_\eps:E(\lambda_N)\to F_\eps$ is bijective, $\{v_i^\eps\}_{i=1}^m$ is a basis of $F_\eps$.  According to Lemma \ref{l:restrict}, $q_\eps(v_i^\eps,v_j^\eps)=r_\eps(u_{N+i-1},v_{N+j-1})$. Let us define the $m\times m$ matrix $A_\eps\equiv[q_\eps(v_i^\eps,v_j^\eps)]_{1\le i,j\le m}$.  We have just seen that $A_\eps$ is the matrix of the quadratic form $r_\eps$ in the orthonormal basis $\{u_{N+i-1}\}_{1\le i\le m}$. Its eigenvalues are therefore $\{\mu_i^\eps\}_{i=1}^m$. 

On the other hand, $A_\eps$ is the matrix of the quadratic form $q_\eps$, restricted to $F_\eps$, in the basis $\{v_i^\eps\}_{i=1}^m$. It follows from Lemma \ref{l:a1} that $v_i^\eps\to u_{N+i-1}$ in $L^2(\Omega)$ as $\eps\to0$, and this means that the basis $\{v_i^\eps\}_{i=1}^m$ is approximately orthonormal. More precisely, if we define the matrix $C_\eps\equiv[\langle v_i^\eps,v_j^\eps\rangle]_{1\le i,j\le m}$, we have that $C_\eps=\mathbb I+E_\eps$, with $\lim_{\eps\to0}E_\eps=0$.  As detailed in Appendix \ref{app:appEVs}, this fact and the estimate $\mu_i^\eps=O\left(\chi_\eps^2\right)$ imply that 
\begin{equation*}
	 \xi_i^\eps =\mu_i^\eps+o(\chi_\eps^2).
\end{equation*}

\section{Connection with the order of vanishing}

In this section, we again denote by
\[E(\lambda_N)=E_1\oplus\dots\oplus E_p\] 
the order decomposition of the eigenspace $E(\lambda_N)$ (see Proposition \ref{prop:DecompES}),  with
\[k_1>\dots>k_p\ge0\]
the associated finite sequence of orders. Our goal is to prove Theorem \ref{thm:orderEVs}. Let us note that our formulation of this theorem does not make any reference to a particular basis of $E(\lambda_N)$. Up to a change of basis, we can therefore assume, in the course of the proof, that  the orthonormal basis $\{u_{N+i-1}\}_{i=1}^m$ has a form which is convenient for our computations. The final result will not depend on this choice of basis. More precisely, we can assume that the orthonormal basis $\{u_{N+i-1}\}_{i=1}^m$ agrees with the order decomposition and diagonalizes each of the quadratic forms $\mathcal Q_j$. Explicitly, this means that, for all $j\in\{1,\dots,p\}$, 
\begin{equation*}
	E_j=\mbox{span}\{u_{N+m_1+\dots+m_{j-1}},\dots,u_{N+m_1+\dots+m_{j-1}+m_j-1}\}
\end{equation*}
and, for all $1\le s<t\le  m_j$,
\begin{equation*}
\mathcal Q_j(u_{N+m_1+\dots+m_{j-1}+s-1},u_{N+m_1+\dots+m_{j-1}+t-1})=0.
\end{equation*}
It follows that, for all $1\le s\le  m_j$,
\begin{equation*}
\mathcal Q_j(u_{N+m_1+\dots+m_{j-1}+s-1},u_{N+m_1+\dots+m_{j-1}+s-1})=\mu_{j,s}.
\end{equation*}

As a first step in the proof of Theorem \ref{thm:orderEVs}, let us find the asymptotic behavior of those eigenvalues $ \nu_i^\eps\equiv \lambda_i^\eps-\lambda_N$  that go most slowly to $0$, that is to say the eigenvalues
\begin{equation*}
	\nu^\eps_{N+m_1+\dots+m_{p-1}},\dots,\nu^\eps_{N+m-1}.	
\end{equation*}
According to Theorem \ref{thm:approxEVs}, we have to find the $m_p$ largest eigenvalues (as $\eps\to0$) of $A_\eps$, the matrix of the quadratic form $r_\eps$ in the basis $\{u_{N+i-1}\}_{i=1}^m$. It follows from Lemma \ref{l:a1} and Corollary \ref{cor:asymptCap} that 
\begin{equation*}
A_\eps=
\left(
\begin{array}{cccccc}
	& &&						&       &						   \\
	&0&&      					&0      &						   \\
	& &&						& 		&						   \\
	& &&\mu_{p,1}\,\rho_{k_p}^\eps&	    &						   \\
	&0&&						&\ddots &						   \\
	& &&						& 		&\mu_{p,m_p}\,\rho_{k_p}^\eps
\end{array}
\right)
+o\left(\rho_{k_p}^\eps\right),
\end{equation*} 
with the functions $\{\rho_k^\eps\}$ defined by Equation \eqref{eq:scale}. Using the min-max characterization of eigenvalues, we conclude that, for $1\le i\le m_p$,
\begin{equation*}
	\mu^\eps_{m_1+\dots+m_{p-1}+i}=\mu_{p,i}\,\rho_{k_p}^\eps+o\left(\rho_{k_p}^\eps\right)
\end{equation*}
and, for $1\le i\le m_1+\dots+m_{p-1}$,
\begin{equation*}
	\mu^\eps_{i}=o\left(\rho_{k_p}^\eps\right).
\end{equation*}
Theorem \ref{thm:approxEVs}, and the fact that $\chi_\eps^2$ and $\rho_{k_p}^\eps$ are of the same order, tell us that the same estimates hold for   the $\nu_i^\eps$'s:  for $1\le i\le m_p$,
\begin{equation*}
	\nu^\eps_{N-1+m_1+\dots+m_{p-1}+i}=\mu_{p,i}\rho_{k_p}^\eps+o\left(\rho_{k_p}^\eps\right)
\end{equation*}
and, for $1\le i\le m_1+\dots+m_{p-1}$,
\begin{equation*}
	\nu^\eps_{N-1+i}=o\left(\rho_{k_p}^\eps\right).
\end{equation*}

The rest of the proof consists of a step-by-step procedure, in which we rescale the quadratic form $q_\eps$ and apply the same arguments in order to identify successive groups of eigenvalue converging to $\lambda_N$ at the same rate. Let us sketch the next step. We set, for $u,v\in\mathcal D_\eps$,
\begin{equation*}
	q_{p-1}^\eps(u,v)\equiv\frac1{\rho_{k_p}^\eps}q_\eps(u,v),
\end{equation*}
and we define the subspace
\begin{equation*}
	F_{p-1}^\eps=\Pi_\eps(E_1\oplus\dots\oplus E_{p-1}).
\end{equation*}

The eigenvalues of $q^\eps_{p-1}$ are $\left\{\frac1{\rho_{k_p}^\eps}\nu_i^\eps\right\}_{i\ge1}$. We know from the first step that, for $1\le i\le m_p$,
\begin{equation*}
	\lim_{\eps\to0}\frac1{\rho_{k_p}^\eps}\nu^\eps_{N-1+m_1+\dots+m_{p-1}+i}=\mu_{p,i}>0.
\end{equation*}
It follows immediately that there exists $\gamma>0$ such that, for $\eps>0$ small enough,
\begin{equation*}
  \left|\frac1{\rho_{k_p}^\eps}\nu_{N-1+i}^\eps\right|\le\gamma \mbox{ for } 1\le i\le m_1+\dots+m_{p-1};
\end{equation*}
\begin{equation*}
	\frac1{\rho_{k_p}^\eps}\nu_{N+m_1+\dots+m_{p-1}}^\eps\ge2\gamma;
\end{equation*}
and, in case $N\ge2$,
\begin{equation*}
	\frac1{\rho_{k_p}^\eps}\nu_{N-1}^{\eps}\le -2\gamma.
\end{equation*}
Repeating the arguments of Section \ref{sec:appLemma}, we can show that for all $v\in F_{p-1}^\eps$ and $w\in \mathcal D_\eps$,
\begin{equation*}
	\left|q^\eps_{p-1}(v,w)\right|\le o\left(\left(\frac{\rho_{k_{p-1}}^\eps}{\rho_{k_p}^\eps}\right)^{1/2}\right) \|v\|\|w\|.
\end{equation*}
Using the arguments in the proof of Theorem \ref{thm:approxEVs} and in the first step, we conclude that, for $1\le i\le m_{p-1}$,
\begin{equation*}
	\frac1{\rho_{k_p}^\eps}\nu^\eps_{N-1+m_1+\dots+m_{p-2}+i}=\mu_{p-1,i}\,\frac{\rho_{k_{p-1}}^\eps}{\rho_{k_p}^\eps}+o\left(\frac{\rho_{k_{p-1}}^\eps}{\rho_{k_p}^\eps}\right)
\end{equation*}
and, for $1\le i\le m_1+\dots+m_{p-2}$,
\begin{equation*}
	\frac1{\rho_{k_p}^\eps}\nu^\eps_{N-1+i}=o\left(\frac{\rho_{k_{p-1}}^\eps}{\rho_{k_p}^\eps}\right).
\end{equation*}
This gives us finally, for $1\le i\le m_{p-1}$,
\begin{equation*}
	\nu^\eps_{N-1+m_1+\dots+m_{p-2}+i}=\mu_{p-1,i}\,\rho_{k_{p-1}}^\eps+o\left(\rho_{k_{p-1}}^\eps\right)
\end{equation*}
and, for $1\le i\le m_1+\dots+m_{p-2}$,
\begin{equation*}
	\nu^\eps_{N-1+i}=o\left(\rho_{k_{p-1}}^\eps\right).
\end{equation*}

Carrying on the procedure for $j$ decreasing from $p-2$ to $1$, we obtain Theorem \ref{thm:orderEVs}.

\section{Same vanishing order:  the case of elliptic holes}

In this section we provide  further applications of the results established above. Let us consider the dimension  to be ${d}=2$ and the multiplicity of $\lambda_{N}$ to be $m=2$. We recall that  $\Omega$ and $\omega$ are open bounded connected subsets of $\mathbb{R}^{2}$ satisfying assumption \eqref{e1}. Moreover,
\[
\Omega_\varepsilon\equiv\Omega\setminus K_\eps
\quad \text{where}\quad K_\eps\equiv \eps\overline\omega
\qquad\quad\forall\varepsilon\in]-\varepsilon_0,\varepsilon_0[\,.
\]

Moreover, 
in this subsection we assume that 
\begin{align}\label{eq:samevanishing}
  &\text{all non-zero functions in } E(\lambda_N) \text{ have the same order of vanishing }k.
\end{align}

We will compute explicitly the $2\times2$ matrix of Corollary \ref{cor:doubleEV} in the case of  elliptic holes and then analyze the special case of a circular hole to prove Corollary \ref{c:1}. To consider the case of elliptic holes, we refer to \cite[Section 9]{AbBoLeMu}.
Let $a>b>0$. We consider the ellipse $\mathcal E_0$ defined as
\[
 \mathcal E_0(a,b)  \equiv \left\{(x,y)\in\mathbb R^2,  \frac{x^2}{a^2}+\frac{y^2}{b^2}<1 \right\} = \Big\{ (x_1,x_2)\in \R^2:\ \frac{x_1^2}{b^2+c^2} + \frac{x_2^2}{b^2} < 1 \Big\},
\]
where $c$ is  the distance between the two foci, which satisfies $c^2=a^2-b^2$. 

For the sake of simplicity and without loss of generality, we perform a change of variables by rotating  the domain,  in such a way that in the new domain, the major axis of the  small elliptic hole is lying along the $x_1$-axis, so that 
\[%\begin{equation} 
\omega = \mathcal E_0(a,b) \label{eq:omegasimnew}.
\]%\end{equation}
Let us now choose $\{u_N,u_{N+1}\}$ an orthonarmal basis of $E(\lambda_N)$. According to assumption \eqref{eq:samevanishing}, we have, for $l\in \{N,N+1\}$, 
\begin{equation}
r^{-k} u_l(r\cos t,r\sin t ) \to \beta_l\sin(k t + k\varphi_l) \qquad \text{ as }r\to 0,\label{eq:orderknew}
\end{equation}
uniformly in $t$ for derivatives of all orders, where $\beta_l\in\R\setminus\{0\}$ and $\varphi_l\in ]-\pi/2k,\pi/2k]$.

We apply Corollary \ref{cor:doubleEV}.
According to \cite[Equation (9.12)]{AbBoLeMu}, the diagonal entries of the matrix $M$ representing $\mathcal Q_1$ in the basis  $\{u_N,u_{N+1}\}$
are
\begin{align}
 &M_{1,1} = \dfrac{-\pi\beta_{N+1}^2 c^{2k}}{2}\,C_{k} \cos(2k\varphi_{N+1}) + \pi \beta_{N+1}^2\, Q_{k}(a,b)\label{eq:m11}\\
 &M_{2,2} = \dfrac{-\pi\beta_{N}^2 c^{2k}}{2}\,C_{k} \cos(2k\varphi_{N}) + \pi \beta_{N}^2\, Q_{k}(a,b),\label{eq:m22}
\end{align}
$C_{k}$ and $Q_{k}(a,b)$ being positive constants depending only on $k$ and $k,\ a,\ b$, respectively (see \cite[Section 9]{AbBoLeMu}). 
{We recall that here $u_N$ and $u_{N+1}$ are not assumed to diagonalize $\mathcal Q_1$, so} we need a similar formula for the mixed term, too. Following all steps in \cite[Subsection 9.1]{AbBoLeMu}, we can compute the first contribution. We introduce the function $F: (\xi,\eta)\mapsto (x_1,x_2)$ which changes the variables  into elliptic coordinates by 
\[%\begin{equation}\label{eq:ellcoord}
\begin{cases}
x_1=c\cosh(\xi)\cos(\eta),\\
x_2=c\sinh(\xi)\sin(\eta),
\end{cases}
\quad \cbk \xi\in [0, +\infty[,\ \eta\in[0,2\pi[.\cbk
\]%\end{equation}
$F$ is a $C^\infty$ diffeomorphism from $\cbk D:= [0,+\infty[\times [0,2\pi[$ onto $\R^2$. It is actually a conformal map, as noted in \cite[Subsection 3.2]{AbFeHiLe2019}. For any $l=N,N+1$ let $W^l_{k}= \mathsf{u}^l_{\overline{k}} \circ F $ and $a_{j,l},\ b_{j,l}$ its Fourier coefficient defined as
\begin{align*}
 &a_{j,l}(\xi)= \frac1\pi \int_0^{2\pi} W^l_{k}(\xi,\eta) \cos(j\eta) \,d\eta \quad \text{ for }j\in\N,\\ 
 &b_{j,l}(\xi)= \frac1\pi \int_0^{2\pi} W^l_{k}(\xi,\eta) \sin(j\eta) \,d\eta \quad \text{ for }j\in\N\setminus\{0\}.
\end{align*}
Following the same computations as in \cite[Subsection 9.1]{AbBoLeMu}, we compute
\begin{align*}
 &\int_{\mathbb{R}^2 \setminus \overline{\omega}}\nabla \mathsf{u}^{N+1}_{\overline{k}}\cdot\nabla\mathsf{u}^{N}_{\overline{k} }\, dt 
 = \pi \sum_{j\geq 1} j\left( a_{j,N}(\bar\xi)\,a_{j,N+1}(\bar\xi) + b_{j,N}(\bar\xi)\,b_{j,N+1}(\bar\xi) \right)\\
 &= \sum_{\begin{array}{c}1\le j\le k\\ \cbk k+j\mbox{ even} \cbk\end{array}}\frac{\pi \beta_{N}\beta_{N+1}c^{2k}}{4^{k-1}}j\left(\begin{array}{c}k\\ \frac{k+j}2\end{array}\right)^2\left(\sin(k\varphi_{N})\sin(k\varphi_{N+1})\cosh^2j\bar\xi\right.\\
 &\left.\qquad \qquad \qquad+\cos(k\varphi_{N})\cos(k\varphi_{N+1})\sinh^2j\bar\xi\right) \\
 &= \sum_{\begin{array}{c}1\le j\le k\\ \cbk k+j\mbox{ even} \cbk\end{array}}\frac{\pi \beta_{N}\beta_{N+1}c^{2k}}{4^{k-1}}j\left(\begin{array}{c}k\\ \frac{k+j}2\end{array}\right)^2\left(\sin(k\varphi_{N})\sin(k\varphi_{N+1})\right.\\
 &\left.\qquad \qquad \qquad+\sinh^2j\bar\xi \big(\cos(k\varphi_{N}-k\varphi_{N+1})\big)\right)\\
 &= \sum_{\begin{array}{c}1\le j\le k\\ \cbk k+j\mbox{ even} \cbk\end{array}}\frac{\pi \beta_{N}\beta_{N+1}c^{2k}}{4^{k-1}}j\left(\begin{array}{c}k\\ \frac{k+j}2\end{array}\right)^2\left(\frac12\cosh 2j\bar\xi \cos(k\varphi_{N}-k\varphi_{N+1})\right.\\
 &\left.\qquad \qquad \qquad - \frac12 \cos(k\varphi_{N}+k\varphi_{N+1}) \right)\\
 &= -\frac{\pi\beta^2 c^{2k}}2C_{k}\cos (k\varphi_{N}+k\varphi_{N})+\pi \beta^2 c^{2{k}}D_{k}(\bar\xi)\cos (k\varphi_{N}-k\varphi_{N})
\end{align*}
where the third equality follows from adding and subtracting the same quantity $\sin(k\varphi_{N})\sin(k\varphi_{N+1})\sinh^2j\bar\xi$ and factorizing; the fourth equality follows from Werner's formula and identity 
\[ \sinh^2 j\bar\xi = \frac12 \cosh 2j\bar\xi - \frac12;\]
the last one follows from analogous steps as in \cite[Subsection 9.1]{AbBoLeMu}.

More easily, following \cite[Subsection 9.2]{AbBoLeMu}, we obtain 
\[
\int_{\omega} \nabla u^{N+1}_{\#,k}\cdot \nabla u^{N}_{\#,k}\, dt
= \dfrac{\pi \beta_{N}\,\beta_{N+1} c^{2k}}{2} \left(\sum_{j=0}^{k} \left(\begin{array}{c}{k}\\ j\end{array}\right)^2 ({k}-2j)\,e^{2({k}-2j)\bar\xi}\right) \cos (k\varphi_{N}-k\varphi_{N})
\]
in place of \cite[Equation (9.10)]{AbBoLeMu}. 

Following \cite[Subsection 9.3]{AbBoLeMu}, we finally obtain
\begin{align*}%\label{eq:m12}
 M_{1,2}&= \dfrac{-\pi{\beta_{N}}{\beta_{N+1}} c^{2k}}{2}\,C_{k} \cos(k\varphi_{N}+k\varphi_{N+1})\notag\\
 &\ + \pi {\beta_{N}}{\beta_{N+1}}\, Q_{k}(a,b)\cos(k\varphi_{N}-k\varphi_{N+1}).
\end{align*}

As already mentioned, the two eigenvalues of the symmetric matrix with entries $M_{1,1}, \,M_{1,2}$ and $M_{2,2}$ coincide if and only if $M_{1,1}=M_{2,2}$ and $M_{1,2}=0$. 

Taking advantage of the computations performed above, we are in position to prove Corollary \ref{c:1}.

\begin{proof}
[Proof of Corollary \ref{c:1}]
 As mentioned in \cite[Remark 9.3]{AbBoLeMu}, we can recover the case of round holes letting $b\to a$. In this way we obtain 
 \begin{align*}
  &M_{1,1} = 2 \pi k a^{2k} \beta_{N+1}^2, \qquad 
    M_{2,2} = 2 \pi k a^{2k} \beta_{N}^2
\\
&M_{1,2}=2\pi k a^{2k} {\beta_{N}}{\beta_{N+1}}\, \cos(k\varphi_{N}-k\varphi_{N+1}).
 \end{align*}
If ${\beta_{N}}={\beta_{N+1}}$, then $M_{1,1}=M_{2,2}$, but condition $M_{1,2}=0$ implies
\[
 \varphi_{N} - \varphi_{N+1}= \pm \frac{\pi}{2k}.
\]
\end{proof}
We note that if $k=1$, the condition in Corollary \ref{c:1} means the two eigenfunctions  do not have perpendicular nodal lines.

\appendix

\section{Order decomposition of an eigenspace}\label{app:DecompES}

Let us recall the setting. We consider $E \equiv E(\lambda_N)$, the eigenspace associated with eigenvalue $\lambda=\lambda_N(\Omega)$ of Problem \eqref{eq:eige}   and we denote the multiplicity by $m$, i.e. $m=\mbox{dim}(E)$. We study the behavior of eigenfunctions at a point $x_0\in\Omega$  and we may assume $x_0=0$ without loss of generality.  Our  main goal is to establish the existence and uniqueness of the order decomposition (Proposition \ref{prop:DecompES}). Let us recall the corresponding result.

\begin{prop} There exists a decomposition of $E$ into a sum of  subspaces
\[E=E_1\oplus\dots\oplus E_p,\] 
orthogonal for the scalar product in $L^2(\Omega)$, and an associated finite decreasing sequence of integers
\[k_1>\dots>k_p\ge0\]
such that, for all $1\le j \le p$, a function in $E_j\setminus\{0\}$ has the order of vanishing $k_j$ at $0$. In addition, such a decomposition is unique.
\end{prop} 

\begin{proof} We use the fact that all eigenfunctions are analytic, and thus any function in $E\setminus\{0\}$ has a finite order of vanishing. For any $k\in \N$, let us define the mapping $\Pi_k:E\to \R_k[X_1,\dots,X_d]$ that associates to a function its Taylor expansion at $ 0$, truncated to order $k$ (here $\R_k[X_1,\dots,X_d]$ is the space of polynomials in $d$ variables of degree at most $k$). Explicitly, using the standard multi-index notation,
\[\Pi_ku=\sum_{ \beta\in \N^d,\,|\beta|\le k}\frac1{ \beta!}D^{ \beta} u(0)X^{ \beta}.\]    
The mapping $\Pi_k$ is clearly linear, and we denote its kernel by $N_k$. We note that the sequence of subspaces $(N_k)_{k\ge0}$ is non-increasing with respect to inclusion: $N_{k+1}\subseteq N_k$. We claim that $N_k=\{0\}$ for $k$ large enough. Indeed, if it were not the case, the vector space 
\[N_\infty\equiv\bigcap_{k\in \N}N_k\]
would not be trivial, while any function in $N_\infty$ has an infinite order of vanishing, contradicting the remark at the beginning of the proof. It will be convenient in what follows to use the convention $N_{-1}\equiv E$.

We now define the finite sequence
\[k_1>k_2>\dots>k_p\ge0\]
as the integers at which there is a jump in the non-increasing sequence  $(\mbox{dim}(N_k))_{k\ge-1}$. More explicitly, the sequence $(k_j)_{1\le j\le p}$ consists of all the integers $k\ge 0$ for which $ \mbox{dim}(N_k)<\mbox{dim}(N_{k-1})$, arranged in decreasing order.
This means that
\[\{0\}= N_{k_1}\subsetneq N_{k_{2}}\subsetneq\dots\subsetneq N_{k_p}\subsetneq E\]
and moreover
\begin{itemize}
	\item for all $1\le  j\le p-1$ and $k_{j+1} \le k <k_{j}$, $N_k=N_{k_{j+1}}$;
	\item for all $k\ge k_1$, $N_k=\{0\}$.
\end{itemize}

Note that if $E$ contains an eigenfunction which does not vanish at $0$, $N_0=\mbox{ker}(\Pi_0)\neq E=N_{-1}$, and therefore $k_p=0$. This motivates our convention $N_{-1}=E$. For convenience, we define the additional value $k_{p+1}\equiv-1$.

We now define, for $1\le j \le p$, 
\[E_{j}\equiv N_{k_{j+1}}\cap N_{k_j}^\perp,\]
so that we have the orthogonal decomposition
\begin{equation}
\label{eqDecompj}N_{k_{j+1}}= N_{k_j}\oplus E_{j}.
\end{equation}
We obtain in this way the orthogonal decomposition 
\[E=E_1\oplus\dots\oplus E_p.\] 
Let us now show that this decomposition satisfies the property: any $u\in E_j\setminus\{0\}$ has order of vanishing exactly $k_j$. We have $E_j\subseteq N_{k_{j+1}}$ and, by definition of the sequence $(k_j)_{1\le j\le p+1}$,  $N_{k_{j+1}}=N_{k_j-1}$, so $E_j\subseteq N_{k_j-1}$. Therefore, any $u\in E_j$ has order of vanishing at least $k_j$. If the order is greater that $k_j$, $u$ belongs to $N_{k_j}$, and by \eqref{eqDecompj} this implies $u=0$. This shows the required property. We have proved existence in the proposition.

 To show uniqueness, let us consider a decomposition 
\[E=E_1\oplus\dots\oplus E_p\] 
and a sequence $(k_j)_{1\le j\le p+1}$ satisfying the hypotheses of the theorem, with the additional term $k_{p+1}=-1$. Then, for all $1\le j\le p+1$,
 \[N_{k_j}=\oplus_{i=1}^{j-1}E_i.\]
 Futhermore, we can check easily that $N_{k}=N_{k_j}$ for all $k_{j+1}\le k< k_j$, with $1\le j\le p$, and that $N_k=\{0\}$ for all $k\ge k_1$. The sequence $(N_k)_{k\ge-1}$ is thus determined by the decomposition, with the $k_j$'s corresponding to the jumps in dimension. The given decomposition therefore coincides with the one constructed in the first part of the proof.
 \end{proof}

Let us conclude by providing an upper bound for the dimension of the subspaces (Proposition \ref{prop:MaxDim}).

\begin{prop} Let $E=E_1\oplus\dots\oplus E_p$, be the order decomposition of the previous Proposition. Then the dimension of $E_j$ is at most the dimension of the space of spherical harmonics in $d$ variables of degree $k_j$.
\end{prop}

\begin{proof} For $1\le j \le k$, we define the mapping $P_j: E_j\to \R_{k_j}[X_1,\dots,X_d]$,  which associate to $u$ its principal part, namely 
\[P_j u\equiv u_\#=\sum_{ \beta\in \N^d,\,|\beta|= k_j}\frac1{ \beta!}D^{ \beta} u(0)X^{ \beta}.\]
According to the classical results on the local behavior of eigenfunctions, recalled in the introduction, $P_j u$ is a harmonic homogeneous polynomial in $d$ variables of degree $k_j$. The mapping $P_j$ is clearly linear.  In order to prove that $P_j$ is injective, let us consider $u, v \in E_j$ such that $P_ju = P_jv$. By
linearity,  $P_j(u-v) = 0$, which means $u-v \in N_{k_j}$. By \eqref{eqDecompj}, this implies that $u-v = 0$.

 The dimension of $E_j$ is therefore equal to that of its image by $P_j$, which is contained in the space of spherical harmonics in $d$ variables of degree $k_j$.\end{proof}

\section{Lemma on small eigenvalues}
\label{app:smallEVs}

We want to approximate the eigenvalues of a quadratic form which are close to $0$. Let us recall the result.

\begin{prop}
Let $(\mathcal H, \|\cdot\|)$ be a Hilbert space and $q$ be a quadratic form, semi-bounded from below (not necessarily positive), with domain $\mathcal D$ dense in $\mathcal H$ and with discrete spectrum $\{ \nu_i \}_{i\geq1}$. Let $\{ g_i \}_{i\geq1}$ be an orthonormal basis of eigenvectors of $q$. Let $N$ and $m$ be positive integers, $F$ an $m$-dimensional subspace of $\mathcal D$ and $\{  \xi_i^F  \}_{i=1}^m$ the eigenvalues of the restriction of $q$ to $F$.

Assume that there exist positive constants $\gamma$ and $\delta$ such that
\begin{itemize}
 \item[(H1)] $ 0<\delta<\gamma/{\sqrt2}$;
  \item[(H2)] for all $i\in\{1,\dots,m\}$, $|\nu_{N+i-1}|\le\gamma$, $\nu_{N+m}\ge \gamma$ and, if $N\ge2$, $\nu_{N-1}\le-\gamma$;
 \item[(H3)] $ |q(\varphi,g)|\leq \delta\, \|\varphi \|\,\|g\|$ for all $g\in\mathcal D$ and $\varphi \in F$. 
\end{itemize}
Then we have
\begin{itemize}
 \item[(i)] $|\nu_{N+i-1}- \xi_i^F |\le\frac{ 4}{\gamma}\delta^2$ for all $i=1,\ldots,m$; 
 \item[(ii)] $\left\|\Pi_N - \mathbb{I} \right\|_{\mathcal L(F,\mathcal H)} \leq   \sqrt2\delta/\gamma$, where $\Pi_N$ is the projection onto the subspace of $\mathcal D$ spanned by $\{g_N,\ldots,g_{N+m-1}\}$. 
\end{itemize}
\end{prop}

\begin{proof} 
Let us prove only the case $N\geq2$. The case $N=1$ is analogous but simpler. We write $E_N \equiv\mbox{span}\{g_N,\dots,g_{N+m-1}\}$. For all $\varphi\in F$, we set $g\equiv\Pi_N\varphi$ and $h\equiv \varphi-g$. 
 Furthermore, we write $h=h_-+h_+$, with
\begin{align*}
	h_-=&\sum_{i=1}^{N-1}c_i g_i;\\
	h_+=&\sum_{i=N+m}^{\infty}c_i g_i.
\end{align*}
Note that in the case $N=1$ there is no $h_-$ in the decomposition of $h$; the proof follows accordingly.
 
The vectors $\varphi$ and $h$ are orthogonal both for the scalar product in $\mathcal H$ and for the quadratic form $q$, so that
\begin{align}
	\label{eq:Hsum} \|\varphi\|^2=&\|g\|^2+\|h\|^2;\\
	\label{eq:Qsum} q(\varphi)=&q(g)+q(h).
\end{align}
 In addition, $h_+$ and $h_-$ are orthogonal both for the scalar product and for $q$, so that
\begin{align}
	\label{eq:Hsumh}\|h\|^2=&\|h_-\|^2+\|h_+\|^2;\\
	\label{eq:Qsumh} q(h)=&q(h_-)+q(h_+).
\end{align}
Note that 
\begin{align*}
	\|h_-\|^2&=\sum_{i=1}^{N-1}|c_i|^2;\\
	\|h_+\|^2&=\sum_{i=N+m}^{\infty}|c_i|^2;\\
	q(h_-)&=\sum_{i=1}^{N-1}\nu_i|c_i|^2\le0;\\
	q(h_+)&=\sum_{i=N+m}^{\infty} \nu_i|c_i|^2\ge0.
\end{align*}

Hypothesis (H2) implies that 
\begin{equation}\label{eq:appB}
\gamma\|h_\pm\|^2\le|q(h_\pm)|.
\end{equation}
Recalling that $\varphi=h+g$, Assumption (H3) implies
\begin{equation*}
	|q(h_\pm)|=|q(\varphi,h_\pm)|\le \delta\|\varphi\|\|h_\pm\|\le \frac1{\sqrt{\gamma}}\delta\|\varphi\|\sqrt{|q(h_\pm)|},
\end{equation*}
from which
\begin{equation*}
	|q(h_\pm)|\le\frac1\gamma\delta^2\|\varphi\|^2.
\end{equation*}
Equation \eqref{eq:appB} thus implies
\begin{equation*}
	\|h_\pm\|\le\frac1\gamma\delta\|\varphi\|,
\end{equation*}
and finally
\begin{equation*}
	\|h\|\le\frac{\sqrt2}\gamma\delta\|\varphi\|.
\end{equation*}

We have proved (ii).

To prove (i), let us first remark that (H1), together with (ii), implies that $\Pi_N:F\to E_N$ is injective, and therefore bijective since $\mbox{dim}(F)= m=\mbox{dim}(E_N)$. We now assume that $\varphi$ (and therefore $g$) is non-zero. Then, using 
Identities \eqref{eq:Hsum}-\eqref{eq:Qsumh},
\begin{align*}
\left|\frac{|q(\varphi)|}{\|\varphi\|^2}-\frac{|q(g)|}{\|g\|^2}\right|=&\left|\frac{(q(g)+q(h))\|g\|^2-q(g)(\|g\|^2+\|h\|^2)}{\|\varphi\|^2\|g\|^2}\right|= \left|\frac{q(h)}{\|\varphi\|^2}-\frac{q(g)}{\|g\|^2}\frac{\|h\|^2}{\|\varphi\|^2}\right|\\
&\le \frac{|q(h)|}{\|\varphi\|^2}+\frac{|q(g)|}{\|g\|^2}\frac{\|h\|^2}{\|\varphi\|^2}\le\frac{|q(h_-)|+|q(h_+)|}{\|\varphi\|^2}+\frac{|q(g)|}{\|g\|^2}\frac{\|h_-\|^2+\|h_+\|^2}{\|\varphi\|^2}.
\end{align*}
Hypothesis (H2) implies that $|q(g)|/\|g\|^2\le \gamma$. We finally find
\begin{equation} 
\label{eq:RQ}
\left|\frac{|q(\varphi)|}{\|\varphi\|^2}-\frac{|q(g)|}{\|g\|^2}\right|\le\frac{4\delta^2}\gamma.
\end{equation}
We recall the min-max characterization of the eigenvalues $ \{\xi_i^F\}_{i=1}^m$:
\begin{equation*}
 \xi_i^F =\min_{ W\in \mathcal F_i}\max_{ \varphi\in W\setminus\{0\}}\frac{q(\varphi)}{\|\varphi\|^2},
\end{equation*}
where $\mathcal F_i$ is the set of $i$-dimensional subspaces of $F$. By construction of $E_N$, the eigenvalues of the restriction of $q$ to $E_N$ are $\{\nu_{i}\}_{i=N}^{ N+m-1}$. They can also be computed by the min-max characterization:
\begin{equation*}
 \nu_{N+i-1}=\min_{ V\in \mathcal E_i}\max_{ g\in V\setminus\{0\}}\frac{q(g)}{\|g\|^2},
\end{equation*}
where $\mathcal E_i$ is the set of $i$-dimensional subspaces of $E_N$. If we combine these characterizations with Inequality \eqref{eq:RQ} and use the fact that $\Pi_N$ maps $\mathcal F_i$ to $\mathcal E_i$ bijectively, we obtain (i).
\end{proof}

\section{ Approximation of eigenvalues}\label{app:appEVs}

Let us recall the situation we are considering. We are studying the restriction of the quadratic form $q_\eps$ to the $m$-dimensional subspace $F_\eps\subseteq \mathcal D_\eps$. We have found a basis 
$ \{v_i^\eps\}_{i=1}^m$ of $ F_\eps$ such that 
\begin{enumerate}
	\item the matrix $A_\eps\equiv[q_\eps(v_i^\eps,v_j^\eps)]$ of $q_\eps$ in the basis $\{v_i^\eps\}$ has eigenvalues $\{\mu_i^\eps\}$, with $\mu_i^\eps=O(\chi_\eps^2)$;
	\item the Gram matrix $C_\eps\equiv[\langle v_i^\eps,v_j^\eps\rangle]$ is of the form $C_\eps=\mathbb I+o(1)$.
\end{enumerate}
We want to show that the eigenvalues $\{ \xi_i^\eps \}$ of $q^\eps_{|F^\eps}$ satisfy
\begin{equation}\label{eq:EVexp}
	\xi_i^\eps=\mu_i^\eps+o\left(\chi_\eps^2\right).
\end{equation}

Let us denote by $\{w_i^\eps\}_{i=1}^m$ the basis of $F_\eps$ obtained from $ \{v_i^\eps\}_{i=1}^m$ by the Gram-Schmidt orthogonalization procedure. It follows from the form of $C_\eps$ that the change-of-basis matrix $P_\eps$ satisfies $P_\eps=\mathbb I+o(1)$ (this can easily be checked by writing down the expression of the $w_i^\eps$'s in terms of the  $v_i^\eps$'s). Let us denote by $B_\eps$ the matrix of $q^\eps_{|F^\eps}$ in the orthogonal basis $ \{w_i^\eps\}_{i=1}^m$. Then $B_\eps$ has eigenvalues $\{\xi_i^\eps\}$. We have
\begin{equation}\label{eq:matrixExp}
	B_\eps=P_\eps^TA_\eps P_\eps=\left(\mathbb I+o(1)\right)^TA_\eps\left(\mathbb I+o(1)\right)=A_\eps+o\left(\chi_\eps^2\right).
\end{equation}
The expansions \eqref{eq:EVexp} follow directly from \eqref{eq:matrixExp} and the min-max characterization of eigenvalues.

\section*{Acknowledgements}
L.~Abatangelo and P.~Musolino are members of the ``Gruppo Nazionale per l'Analisi Matematica, la Probabilit\`a e le loro Applicazioni'' (GNAMPA) of the ``Istituto Nazionale di Alta Matematica'' (INdAM). They are partially supported by the GNAMPA-INdAM 2020 project ``Analisi e ottimizzazione asintotica per autovalori in domini con piccoli buchi''.

C.~Léna acknowledges the support of COST (European Cooperation in Science and Technology) through the COST Action CA18232 MAT-DYN-NET.

P.~Musolino acknowledges the support of the Project BIRD191739/19 ``Sensitivity analysis of partial differential equations in
the mathematical theory of electromagnetism'' of the University of Padova and of the grant ``Challenges in Asymptotic and Shape Analysis - CASA''  of the Ca' Foscari University of Venice.

\end{document}